\documentclass[11pt, reqno]{amsart}
\usepackage{amsmath, amsthm, amscd, amsfonts, amssymb, latexsym, graphicx, color, stmaryrd}
\usepackage[all,cmtip]{xy}	
\usepackage[bookmarksnumbered, colorlinks, plainpages]{hyperref}
\usepackage{slashed}

\def\presuper#1#2%
  {\mathop{}%
   \mathopen{\vphantom{#2}}^{#1}%
   \kern-\scriptspace%
   #2}

\newcommand{\Nn}{\mathbb{N}}
\newcommand{\Rr}{\mathbb{R}}
\newcommand{\Cc}{\mathbb{C}}
\newcommand{\Zz}{\mathbb{Z}}

\newcommand{\B}{\mathcal{B}}
\newcommand{\dd}{\mathbf{d}}	
\newcommand{\D}{\mathcal{D}}	
\newcommand{\DV}{\presuper{\V}{\mathcal{D}}}
\newcommand{\Dd}{\mathbb{D}}	
\newcommand{\Dirac}{\slashed D}	
\newcommand{\Ee}{\mathbb{E}}	
\newcommand{\Exp}{\mathrm{Exp}}	
\newcommand{\F}{\mathcal{F}}
\newcommand{\K}{\mathcal{K}}	
\renewcommand{\H}{\mathcal{H}}	
\renewcommand{\L}{\mathcal{L}}	
\newcommand{\C}{\mathcal{C}}	
\newcommand{\Cl}{\mathrm{Cl}}	
\newcommand{\A}{\mathcal{A}} 	
\newcommand{\U}{\mathcal{U}}  	
	
\newcommand{\Aroof}{\hat{A}(\nabla)} 

\renewcommand{\S}{\Symb}
\newcommand{\Pp}{\mathbb{P}} 

\newcommand{\Tau}{\mathcal{T}}
\newcommand{\id}{\mathrm{id}}
\newcommand{\SV}{\mathrm{\presuper{\V}{S}}}	
\newcommand{\LV}{\mathrm{\presuper{\V}{\mathcal{L}^1}}} 
\newcommand{\indV}{\mathrm{\presuper{\V}{\ind}}}  
\newcommand{\Tr}{\mathrm{Tr}}			
\newcommand{\TrV}{\mathrm{\presuper{\V}{\Tr}}}	
\newcommand{\TrsV}{\mathrm{\presuper{\V}{\Tr_s}}} 
\newcommand{\etaV}{\mathrm{\presuper{\V}{\eta}}} 
\newcommand{\VOmega}{\mathrm{\presuper{\V}{\Omega}}}
\newcommand{\SigmaV}{\mathrm{\presuper{\V}{\Sigma}}}	
\newcommand{\Spin}{\mathrm{Spin}}
\newcommand{\SO}{\mathrm{SO}}
\newcommand{\SU}{\mathrm{SU}}
\newcommand{\tr}{\mathrm{tr}}			

\renewcommand{\H}{\mathcal{H}} 
\newcommand{\Ff}{\F_{\mathrm{f}}}	
\newcommand{\G}{\mathcal{G}}	
\newcommand{\Gad}{\G^{ad}}	
\newcommand{\Gop}{\mathrm{\G^{(0)}}} 
\newcommand{\Gpull}{\mathrm{\G^{(2)}}}	

\newcommand{\M}{\mathcal{M}}		

\renewcommand{\L}{\mathcal{L}} 
\newcommand{\R}{\mathcal{R}}
\newcommand{\Rk}{\widetilde{\R}}
\renewcommand{\S}{\mathcal{\mathbf{S}}}	

	% ...

\newcommand{\V}{\mathcal{V}}		
		% ...
\newcommand{\End}{\mathrm{End}}		
\newcommand{\Hom}{\mathrm{Hom}}		
\renewcommand{\P}{\mathcal{P}}		
\newcommand{\Diff}{\mathrm{Diff}}	
\renewcommand{\O}{\mathcal{O}}

\newcommand{\HC}{\mathrm{HC}}
\newcommand{\HP}{\mathrm{HP}}

\newcommand{\brel}{b^{\mathrm{rel}}}
\newcommand{\relb}{b_{\mathrm{rel}}}
\newcommand{\Brel}{B^{\mathrm{rel}}}

\newcommand{\scal}[2]{\langle #1, #2 \rangle}	

\newcommand{\op}{\operatorname{op}} 		
\newcommand{\ind}{\operatorname{ind}}

\renewcommand{\deg}{\operatorname{deg}}

\newcommand{\iso}{\xrightarrow{\sim}}	

\newcommand{\nablaslash}{\nabla\hspace{-.8em}/\hspace{.3em}}

\newcommand*{\avintA}{\mathop{\ooalign{$\int_{\A^\ast}$\cr$-$}}}

\newcommand{\avintVA}{\mathrm{\presuper{\V}{\avintA}}} 

\newcommand*{\davint}{\mathop{\ooalign{$\displaystyle \int$\cr$-$}}}
\newcommand*{\davintA}{\mathop{\ooalign{$\displaystyle \int_{\A^\ast}$\cr$-$}}}
\newcommand*{\davintM}{\mathop{\ooalign{$\displaystyle \int_M$\cr$-$}}}

\newcommand{\davintVA}{\mathrm{\presuper{\V}{\davintA}}} 
\newcommand{\davintVM}{\mathrm{\presuper{\V}{\davintM}}} 
\newcommand{\davintV}{\mathrm{\presuper{\V}{\davint}}} 

\newtheorem{Thm}{Theorem}[section]
\newtheorem{Lem}[Thm]{Lemma}
\newtheorem{Prop}[Thm]{Proposition}
\newtheorem{Cor}[Thm]{Corollary}
\theoremstyle{definition}
\newtheorem{Def}[Thm]{Definition}

\newtheorem{Exa}[Thm]{Example}

\newtheorem{Rem}[Thm]{Remark}

\begin{document}
\setcounter{page}{1}


\title[Getzler rescaling]{Getzler rescaling via adiabatic deformation and a renormalized local index formula}


\author[Karsten Bohlen, Elmar Schrohe]{Karsten Bohlen, Elmar Schrohe}

\address{$^{1}$ Bergische Universit\"at Wuppertal, Germany}
\email{\textcolor[rgb]{0.00,0.00,0.84}{bohlen@uni-wuppertal.de}}

\address{$^{2}$ Leibniz Universit\"at Hannover, Germany}
\email{\textcolor[rgb]{0.00,0.00,0.84}{schrohe@math.uni-hannover.de}}





\maketitle

\begin{abstract}
We prove a local index theorem of  Atiyah-Singer type for Dirac operators on manifolds with a Lie structure at infinity (\emph{Lie manifolds} for short). 
With the help of a renormalized supertrace, defined on a suitable class of regularizing operators, 
the proof of the index theorem relies on a rescaling technique similar in spirit to Getzler's rescaling. With a given Lie manifold we associate an appropriate integrating Lie groupoid. We then describe the heat kernel of a geometric Dirac operator via a functional calculus with values in the convolution algebra of sections of a rescaled bundle over the adiabatic groupoid.
Finally, we calculate the right coefficient in the heat kernel expansion by deforming the Dirac operator into a polynomial coefficient operator over the rescaled bundle and applying the Lichnerowicz theorem to the fibers of the groupoid and the Lie manifold.
\end{abstract} \maketitle

\section{Introduction}

There are various routes to the Atiyah-Singer index theorem (cf. \cite{AS}, \cite{ASI}, \cite{ASIII}) for the Fredholm index of elliptic operators
on a closed manifold. 
Different proofs in turn have often given rise to profound generalizations, in particular to the index theory of elliptic operators on non-compact manifolds modeled on manifolds with singularities,
manifolds with boundary or manifolds with corners.
A particularly fruitful approach is based on the deformation groupoid (the \emph{tangent groupoid}) introduced by A. Connes, \cite{C}. 
It has given rise to a number of extensions, see e.g. \cite{CLM}, \cite{DLN}, \cite{MN}.  
In the analysis of non-compact manifolds modeling  different types of singular manifolds, Lie groupoids enter naturally as 
models for singular spaces, an observation first made by A. Connes. 
The problem then is to find ellipticity conditions implying the Fredholm property of a suitable class of differential operators acting between appropriate Sobolev spaces as  the most natural condition, namely the pointwise invertibility of the invariantly defined principal symbol, is no longer sufficient. 
If the noncompact manifold is the interior of a compact manifold with corners and the boundary strata are embedded submanifolds of the same dimension, the index theory of foliations initiated by A. Connes and G. Skandalis \cite{CS} provides a basis for the formulation of an index problem. In general, however, the dimension of the strata will vary. 
Moreover, Connes realized that the natural receptacle for the foliation index 
is the $K$-theory of the $C^{\ast}$-algebra of the holonomy groupoid of the foliation.
Similarly, for a manifold with corners, the corresponding generalized analytic index  maps into the $K$-theory of the $C^{\ast}$-algebra of the Lie groupoid which desingularizes the manifold: $\ind_a \colon K(\A^{\ast}( \G ) ) \to K(C_r^{\ast}(\G))$. 
The task therefore is to find a purely topological interpretation of the generalized analytic index. This has been achieved for several cases of singular manifolds, see e.g. \cite{MN} for Lie manifolds.
A significant drawback of this strategy is that the generalized analytic index almost never equals the Fredholm index. In fact, both agree for closed manifolds, since in this case the groupoid under consideration is  the pair groupoid whose  $C^{\ast}$-algebra is the algebra of the compact operators.
In other interesting cases the Fredholm index does not equal the generalized analytic index. 
The more difficult problem is therefore to calculate the \emph{Fredholm index} in topological terms, thus generalizing the Atiyah-Singer index theorem to
a large class of non-compact manifolds.

In this article we consider \emph{manifolds with a Lie structure at infinity} or \emph{Lie manifolds} for short.
While many special instances of these manifolds have been studied in the literature and index theorems of the above type have been proven by different techniques, a general Fredholm index theorem, valid for any Lie manifold, has not yet been obtained. We refer to the excellent survey \cite{Nsurv} for more information. The problem 
lies in the more complicated Fredholm conditions on non-compact manifolds and the fact that the boundary strata  give rise to non-local invariants in the resulting index theorem. Only in the simplest case of asymptotically flat Lie structures is a direct analog of Atiyah-Singer possible, cf. \cite{CN}. 

We will follow the strategy to first establish a \emph{local index theorem} via the heat kernel and then use this local index theorem to prove the Atiyah-Singer index formula.
On the other hand we adhere to the program, started by A. Connes and continued by other authors, of using deformation groupoids in order to extract the Fredholm index
and to express it in topological terms.
The particular technique, however,  is different from the tangent groupoid proof in \cite{C}, because this proof would a priori only calculate the generalized analytic index. 
(Note, however, that at least for manifolds with boundary, the authors in \cite{CLM} have obtained a groupoid version of the Atiyah-Patodi-Singer index theorem by modifying Connes' technique.) 
We describe instead a proof which combines the rescaling technique of Getzler with the adiabatic groupoid.

Early proofs of the local index theorem are due to Atiyah-Bott-Patodi \cite{ABP}, Gilkey \cite{Gilk} and Patodi \cite{P}. 
Our proof is based on Getzler's rescaling proof, see \cite{G} and also \cite{BGV} for a very good exposition. We think that it is possible to use the idea of Getzler of replacing
the heat kernel $k=k(t,x)$ by a \emph{rescaled} heat kernel $u^{\frac{n}{2}} k(ut, u^{\frac{1}{2}} x)$, $0<u\le1$, subsequent calculation of the asymptotic expansion of the rescaled kernel
and application of the Lichnerowicz theorem in the limit $u \to 0^+$, and adapt it to our more general case. Nevertheless, we have 
chosen to apply a deformation groupoid argument. The idea for such an argument in the standard case, using the tangent groupoid, 
can be found already in Quillen's notebooks, \cite{Q}. 
We partly rely on unpublished notes by P. Siegel \cite{S} and the expository account of Getzler's argument by J. Roe \cite{R}.
Siegel gives an account of a rescaling technique using the tangent groupoid, deriving the local index formula for a  smooth closed manifold.
In our more general case one has to confront a number of difficulties which we will explain in the sequel.

\subsection{Overview}
\subsection*{Lie manifolds}
Manifolds with a Lie structure at infinity have been introduced by Ammann, Lauter and Nistor, \cite{ALN}. They can be used to model many types of singular manifolds. 
A Lie manifold is a triple $(M, \A, \V)$, where $M$ is a compact $n$-dimensional manifold with corners and
$\V \subset \Gamma(TM)$ is a Lie algebra of smooth vector fields. 
Moreover, $\V$ is assumed to be a subalgebra of the Lie algebra $\V_b$ of all vector fields tangent to the boundary strata  
and a finitely generated  projective $C^{\infty}(M)$-module.
Also the compact manifold with corners $M$ is thought of as a compactification of a non-compact manifold with a degenerate,
singular metric which is of product type at infinity. 
We denote by $\partial M$ the (stratified) boundary of $M$ and by  $M_0 = M \setminus \partial M$ the interior.
By the Serre-Swan theorem there exists a vector bundle $\A \to M$ such that $\Gamma(\A) \cong \V$.
The bundle $\A$ has the structure of a Lie algebroid. A further piece of information we need is a Lie groupoid $\G \rightrightarrows M$.
It is known that for any Lie structure there is an $s$-connected Lie groupoid $\G$ such that $\A(\G) \cong \A$. For general Lie algebroids, Crainic and Fernandes \cite{CF} obtained computable obstructions for the integrability.  
%
%
%

\subsection*{Assumptions.} 
The Lie manifold $(M, \A, \V)$  is \emph{spin}, i.e. there is a spin structure $S \to M$, cf. \cite{ALN2};
 $\A_{|M_0} \cong TM_0$ and $\G_{|M_0} \cong M_0 \times M_0$ are the tangent bundle
and pair groupoid on the interior, respectively. 
We also assume that  the Lie manifold is  \emph{non-degenerate}, i.e., it is \emph{renormalizable}, which means there exists a renormalized trace, see Definition \ref{VTr}, and we can find
an integrating Lie groupoid which is Hausdorff with a smooth heat kernel. We recall particular examples of such groupoids in the main body of the paper.

\subsection*{Approach}
We let $W$ be a $\Cl(\A)$-module, where $\Cl(\A) \to M$ denotes the bundle of Clifford algebras on the fibers of $\A$.
By $D$ we denote a geometric Dirac operator obtained from an \emph{admissible} connection $\nabla^W$, cf. \cite{LN}.
We choose a positive bilinear form $g = g_{\A}$ on $\A$, called a \emph{compatible metric}, cf. \cite{ALN2}.

The heat kernel $\kappa_t$ of $e^{-t D^2}$ will not be of trace class in general.
We therefore introduce the \emph{renormalized super-trace} $\TrsV$ which relies on a renormalization at infinity. 
Since we later apply the renormalized trace to the heat kernel we need the assumption that the heat kernel is smooth.
In addition, we introduce a suitable class $\S(\G)$ of rapidly decaying functions or distributions over the integrating groupoid and a corresponding class
$\SV(M)$ over the Lie manifold. 
Under our assumptions we prove that  $\SV(M)$ can be identified with $\S(\G)$ via the \emph{vector representation} $\varrho \colon \End(C^{\infty}(\G)) \to \End(C^{\infty}(M))$.
The vector representation is characterized by the equality: $(\varrho(P) f) \circ r = P(f \circ r)$, where $r$ is the range map of the groupoid (a surjective submersion),
$P \in \End(C^{\infty}(\G))$ and $f \in C^{\infty}(M)$, see also \cite{ALN}, \cite{NWX}.
We often make use of the notation $C^{0,\infty}(\G)$ which denotes the class of functions over $\G$ which are continuous on $\G$ and smooth on each $s$-fiber $\G_x = s^{-1}(x), \ x \in M$.

In the classical setting, the tangent groupoid deforms the pair groupoid over the manifold $M$ into the tangent bundle $TM$. In our case we deform the integrating groupoid $\G$, rather than just the pair groupoid, and consider the \emph{adiabatic groupoid}
$\Gad = \G \times (0,1] \cup \A(\G) \times \{0\}$ which deforms $\G$ into the Lie algebroid $\A(\G)$.
Then we perform the rescaling over the adiabatic groupoid adapted to a formal Ansatz for the asymptotic expansion of the heat kernel.
The geometric admissible Dirac operator $D$ on a Lie manifold is realized as the vector representation of a corresponding geometric admissible
Dirac operator $\Dirac$ on the Lie groupoid, see \cite{LN}.
The heat kernel $k_t$ on the Lie groupoid has the vector representation $\kappa_t$ which is the heat kernel on the Lie manifold.
We also rely on \cite{So} and \cite{So2} for the proof of the approximation of the heat kernel on Lie groupoids as described
for complete Riemannian manifolds in \cite{BGV}. We use this approximation and the estimates from \cite{So} to show that
the heat kernel is contained in the Schwartz class $\S(\G)$.
The asymptotic expansion Ansatz is 
$e^{-t \Dirac^2} \sim (4\pi)^{-\frac{n}{2}} t^{-\frac{n}{2}} \sum_{i=0}^{\infty} a_i t^i$. 
We next describe a way to extract the coefficient $a_{n/2}$ in the asymptotic expansion of the heat kernel.
For this we deform $\Dirac$ into a smooth equivariant family of operators on the Lie algebroid $\A(\G)$ associated to $M$.
The rescaling deforms $\Dirac$ in such a way that the Clifford multiplication is taken into account and at the same
time the right coefficient in the Ansatz is extracted.
This is done by a rescaling of the Clifford algebra. The result is that $\Dirac$ is deformed into a polynomial coefficient
operator whose supertrace has the right asymptotics.
We then study the groupoid convolution algebra $\S(\Gad, \Hom(W))$ where $\hom(W) \to M$, given by $\hom(W)_x \cong \hom(W_x, W_x) \cong \Cl(\A_x \otimes \Cc) \otimes \End_{\Cl}(W_x)$, is lifted to an equivariant bundle $\Hom(W) \to \Gad$.
Given a Clifford filtration by degree $\Cl_0 \subseteq \Cl_1 \subseteq \cdots \subseteq \Cl(\A \otimes \Cc)$, we can extend
this filtration to a neighborhood of $\A$ within the adiabatic groupoid $\Gad$. Here we view $\A$ as an embedded boundary stratum of the manifold $\Gad$ and
use the accompanying tubular neighborhood to extend the filtration.
Subsequent to this we introduce an equivariant \emph{rescaled bundle} $\Ee \to \Gad$ extending $\Hom(W)$ such that the sections of this bundle
have a polynomial coefficient expansion, where the coefficients are contained in the sections of the extended Clifford filtration.
We define a functional calculus which realizes the groupoid heat kernel as an element of the convolution algebra of rapidly decaying sections
of the rescaled bundle $\S(\Gad, \Ee)$.
The final calculation of the coefficient in the asymptotic expansion relies on the Lichnerowicz theorem applied to the fibers of the Lie groupoid and, by $\G$-invariance,
to the Lie manifold.

\subsection{The main theorems}
We will first prove the following result: 

\begin{Thm}
Let $(M, \A, \V)$ be an $n$-dimensional non-degenerate Lie manifold, $S \to M$ a spin structure, $\Cl(\A) \to M$ the bundle of Clifford algebras and 
$W \in \Cl(\A)-mod$ a Clifford module.
Given a compatible Riemannian metric $g = g_{\A}$ fix an admissible connection $\nabla^W$ and the corresponding Dirac operator $D = D^W \in \Diff_{\V}^1(M; W)$.
Then we have the formula for the renormalized index
\begin{align}\
& \indV(D):=\lim_{t \to \infty} \TrV_s(e^{-t D^2}) = \davintV \Aroof \wedge \exp F^{W / S} \,d\mu + \etaV(D) \label{indthm}
\end{align}
where $F^{W / S}$ is the twisting curvature and $\Aroof$, for the curvature tensor $R$ obtained from the compatible metric, denotes the $n$-form given by the formal power series
\[
h(R) = \left(-\frac{i}{2\pi}\right)^{\frac{n}{2}} \det\left(\frac{\frac{1}{2} R}{\sinh(\frac{1}{2} R)}\right)^{\frac{1}{2}}.
\]

The function $\etaV$ is the renormalized $\eta$-invariant which is given by the integrated trace defect
\[
\etaV(D) := \frac{1}{2} \int_{0}^{\infty} \TrV_s([D, D e^{-t D^2}]) \,dt. 
\]

\label{Thm:locind}
\end{Thm}

The left hand side of \eqref{indthm} has been shown to converge to the Fredholm index in special cases (cf. \cite{M}, Section 7.8). We discuss a general 
criterion for the equality of these indices below. The \emph{trace defect} $\etaV$ on the right hand side can be calculated in terms of restrictions to the boundary strata (cf. \cite{M}, Section 5.5). We refer to \cite{L1}, \cite{L2}, \cite{LMP2} and \cite{M} for the discussion in the case of $b$-manifolds.
A local index formula in the special case of cusp vector fields has been obtained in \cite{LM} and for the case of a fibered cusp Lie structure
in \cite{LM2} as well as in \cite{LMP} by using the method of deformation of the metrics of b-, cusp and fibered cusp type.
We also refer to \cite{MR} for a $K$-theoretic index theorem on manifolds with fibered cusp structure.
Note that $\indV(D)$ is defined  independent of the Fredholm property of $D$ and might in general not be integer valued.
The following Theorem gives conditions on the integrating Lie groupoid of a given Lie manifold which are sufficient to obtain equality
of the Fredholm index with the renormalized index, whenever the Dirac operator under consideration is Fredholm.

\begin{Thm}\label{1.2}
Let $(M, \A, \V)$ be a non-degenerate spin Lie manifold for which there exists an integrating Lie groupoid $\G \rightrightarrows M$ that is strongly amenable and has polynomial growth.
Then for any admissible geometric Dirac operator $D = D^W$ on $M$ which is \emph{fully elliptic} we have $\ind(D) = \indV(D)$.
\label{Thm:Fh}
\end{Thm}

A pseudodifferential operator in the Lie calculus (cf.~\cite{ALN}) is \emph{fully elliptic} whenever its principal symbol and indicial symbol are both pointwise invertible.
In particular any geometric admissible Dirac operator on a spin Lie manifold is fully elliptic if its indicial symbol is pointwise invertible.
Under the above conditions on the integrating Lie groupoid, full ellipticity is equivalent to the operator being Fredholm, see e.g.~\cite{N}. 
We will prove Theorem \ref{1.2} in the final section. There we will also give geometric applications of the index formula to the existence of compatible metrics with positive scalar curvature on
Lie manifolds.


The paper is organized as follows. In the second section we give the definition of the geometric Dirac operators for Lie groupoids
and Lie manifolds. We also prove the Lichnerowicz theorem for the generalized Laplacian on a Lie manifold defined with respect to an admissible connection.
In the third section we study the groupoid heat kernel and its approximation. We introduce a class of rapidly decaying functions
on a Lie groupoid and show that under suitable conditions the heat kernel is contained in this class.
In Section four we define a functional calculus for the convolution algebra over the adiabatic groupoid.
The fifth section contains the definition of the renormalized super trace on Lie manifolds as well as the class of rapidly 
decaying functions on Lie manifolds. Then, in section six we introduce the rescaling and prove the main theorem.
Finally, we discuss Fredholm conditions and the Fredholm index in the last section. As an application of the index theorem we study
obstructions to compatible metrics with positive scalar curvature.

\section{Dirac operators on Lie manifolds}

\label{section2}



Geometric Dirac operators on Lie manifolds are given as vector representations 
of operators on Lie groupoids integrating the Lie structure.  
In this section we will outline some details of their construction, following \cite{ALN}, and state the corresponding Lichnerowicz theorem.  
To this end we will introduce the notion of an \emph{admissible} connection associated with the spin structure.

Denote by $P_{\SO}(\A) \to M$ the bundle of oriented orthonormal frames. This is a principal $\SO(n)$-bundle.
According to \cite{ALN2}, a  \emph{spin structure} over $M$ is a tuple $(P_{\Spin}(\A), \alpha)$, where $P_{\Spin}(\A)$ is a principal $\Spin(n)$-bundle 
and $\alpha \colon P_{\Spin}(\A) \to P_{\SO}(\A)$
is a fiber map over the identity of $M$, compatible with the double covering $\theta \colon \Spin(n) \to \SO(n)$ and the corresponding group actions, i.e., the following diagram commutes
\[
\xymatrix{
\Spin(n) \times P_{\Spin}(\A) \ar[dd]_{\theta \times \alpha} \ar[r] & P_{\Spin}(\A) \ar[dd]_{\alpha} \ar[rd] & \\
& & M ,\\
\SO(n) \times P_{\SO}(\A) \ar[r] & P_{\SO}(\A) \ar[ur] & 
}
\]
where the horizontal arrows are induced by the  group actions.

The spinor bundle is defined as $S := P_{\Spin}(\A) \times_{\sigma_n} \Sigma_n$, where $\sigma_n \colon \Spin(n) \to \SU(\Sigma_n)$
is the complex spinor representation (i.e., the restriction of an odd complex irreducible representation of the Clifford algebra
on $n$-dimensional space). 
Here $\Sigma_n$ denotes an irreducible spin-representation of $\Cl_n(\A) \otimes \Cc$. If $n$ is odd there are two distinct
irreducible representations. For $n$ even, there is one irreducible representation which splits
into two non equivalent sub-representations. See also Lawson and Michelsohn \cite[Section II.3]{LaMi}.

A Clifford module  $W \to M$ is a complex vector bundle together with a positive definite inner product $\scal{\cdot}{\cdot}$, anti-linear in the second component,  
an $\A^\ast$-valued connection $\nabla^W\in \Diff_\V(M,W, W\otimes \A^\ast)$, the space of $\V$-differential operators, acting between $W$ and $W\otimes \A^\ast$, 
 and a linear bundle map $c \colon\A \otimes W \to W, \  X \otimes \varphi \mapsto X \cdot \varphi$
called Clifford multiplication, such that the following holds. 
\begin{enumerate}
\item 
$(X \cdot Y + Y \cdot X + 2 g(X, Y)) \cdot \varphi = 0$ for each $X, Y \in \Gamma(\A), \ \varphi \in \Gamma(W)$.
\item 
$\nabla^W$ is  \emph{metric} 
\[
\partial_X \scal{\psi}{\varphi} = \scal{\nabla_X^W \psi}{\varphi} + \scal{\psi}{\nabla_X^W \varphi}, \ X \in \Gamma(\A), \ \varphi, \psi \in \Gamma(W).
\]
\item 
Clifford multiplication with vectors satisfies
\[
\scal{X \cdot \psi}{\varphi} = \scal{\psi}{X\cdot \varphi}, \ \varphi, \psi \in \Gamma(W), \ X \in \Gamma(\A).
\]

\item 
The connection is \emph{admissible}, i.e.
\[
\nabla_X^W(Y \cdot \varphi) = (\nabla_X Y) \cdot \varphi + Y(\nabla_X^W \varphi), \ X, Y \in \Gamma(\A), \ \varphi \in \Gamma(W).
\]
Here $\nabla$ is the Levi-Civita connection with respect to the compatible metric.
We also assume that $W$ is $\mathbb Z_2$-graded, $W=W^+\oplus W^-$ and the grading is compatible with the Clifford action, i.e. 
\begin{eqnarray*}
c(\Cl(\A)^+ ) W^\pm \subseteq W^\pm &\text{ and }&c(\Cl(\A)^- ) W^\pm \subseteq W^\mp. 
\end{eqnarray*}
\end{enumerate}
\begin{Def}
Let $W \to M$ be a Clifford bundle and $g$ a compatible metric. Then the geometric Dirac operator $D$ is defined by
the composition $D= c \circ (\id \otimes \sharp) \circ \nabla^W$, 
acting on $\Gamma(W)$, 
\[
\xymatrix{
\Gamma(W) \ar[r]^-{\nabla^W} & \Gamma(W \otimes \A^{\ast}) \ar[r]^-{\id \otimes \sharp} & \Gamma(W \otimes \A) \ar[r]^-{c} & \Gamma(W),
}
\]
where $c$ denotes Clifford multiplication and $\sharp$ is the conjugate-linear isomorphism $\A \cong \A^{\ast}$ induced by the metric $g$.
\label{Def:Dirac1}
\end{Def}
Note that $c$ is a $\V$-operator of order $0$ and $\nabla^{W}$ is a $\V$-operator of order $1$, hence $D$ is in $\Diff_{\V}^1(M; W)$.
The principal symbol of $D$ satisfies $\sigma_1(D) \xi = i c(\xi) \in \End(W)$, hence it is invertible for $\xi \not= 0$, and $D$ is elliptic.\medskip

Following \cite{LN}, we next outline the construction of a geometric Dirac operator $\Dirac$ on $\G$ as a $\G$-invariant family of operators on the $s$-fibers $(\G_x)_{x \in M}$
of a given Lie groupoid $\G \rightrightarrows M$ integrating the Lie structure, i.e. $\A(\G) \cong \A$. 
Fix the spinor bundle $S \to M$ as above,  the bundle $\Cl(\A) \to M$  of Clifford algebras 
and a Clifford module $W \to M$.  The Clifford multiplication defines a map $c \colon \Cl(\A) \to \End(W)$.

The Levi-Civita connection on $\G$  is obtained as follows.
Let $X \in \Gamma(\A)$ and let $\tilde{X}$ denote a lift to a $\G$-invariant $s$-vertical vector field (i.e. a smooth section of $T_s \G := \ker ds$).
For $x \in M$ let $g_x$ be the  metric on $\G_x$ induced by the fixed compatible Riemannian metric $g$.
Denote by $\nabla^x \colon \Gamma(T_s \G_x) \to \Gamma(T_x \G_x \otimes T_s^{\ast} \G_x)$ the Levi-Civita
connection associated to $g_x$.
We obtain a smooth and $\G$-invariant family of differential operators 
$\nabla_{\tilde{X}}^x \colon \Gamma(T_s \G_x) \to \Gamma(T_s \G_x)$ that descends to $\nabla_X \in \Diff(\G, \A)$. 

According to  \cite[Proposition 6.1]{LN}, we find a $\G$-invariant connection $\nablaslash^W$ on $\G$, which descends to the previously defined connection $\nabla^W$ on $M$ and  satisfies the following condition of \emph{admissibility}, 
\begin{align}
& \nablaslash_X^W(c(Y) \xi) = c(\nabla_X Y) \xi + c(Y) \nablaslash_X^W(\xi), \ \xi \in \Gamma(r^{\ast} W), \ X, Y \in \Gamma(r^{\ast} \A).
\end{align}

\begin{Def}
Let $W \to M$ be a Clifford module and $\G \rightrightarrows M$ a Lie groupoid.
The geometric Dirac operator $\Dirac^W$ is defined by $\Dirac^W := c \circ (\id \otimes \sharp) \circ \nablaslash^W$ where
$\sharp$ denotes the conjugation isomorphism induced by the fixed compatible metric $g$, $c \in \Hom(W \otimes \A^{\ast})$
Clifford multiplication and $\nablaslash^W \in \Diff(\G; r^{\ast} W, r^{\ast} W \otimes \A^{\ast})$ an admissible connection.
\label{Def:Dirac2}
\end{Def}

It is shown in \cite{LN} that, with the above definition, the Dirac operator on the Lie manifold $M$ is the vector representation of the Dirac operator on the Lie groupoid. We now state the Lichnerowicz theorem comparing the generalized Laplacian $D^2$ of a Dirac operator $D$ on a Lie manifold to the Laplacian $\Delta^W$ associated with the admissible connection: 
$$\Delta^W = - \sum_{i,j} g^{ij} (\nabla_i^W \nabla_j^W - \sum_k \Gamma_{ij}^k \nabla_k^W) . $$
The proof is similar as in \cite{BGV}, but
we provide the details to make the paper more self-contained. 
We start with the following observation: 

\begin{Lem}
The curvature $(\nabla^{W})^2 \in \Lambda^2(\End(W))$ decomposes under the isomorphism
\[
\End(W) \cong \Cl(\A^{\ast}) \otimes \End_{\Cl(\A^{\ast})}(W)
\]
as $R^W + F^{W / S}$, where $R^W$ is the action of the Riemannian curvature $R$ on $W$ given by
\begin{eqnarray}\label{Thm:Lichnerowicz.3}
R^W(e_i, e_j) = \frac{1}{4} \sum_{k,l = 1}^{n} g(R(e_i, e_j) e_k, e_l) c(e^k) c(e^l), 
\end{eqnarray}
for an arbitrary orthonormal frame $\{e_1, \cdots, e_n\}$ of $\A$ and the dual frame $\{e^1, \cdots, e^n\}$ and $F^{W / S}= (\nabla^{W})^2 - R^W \in \Lambda^2(\End_{\Cl(\A^\ast)} W)$. 
\label{Lem:twisting}
\end{Lem}

\begin{proof} Note that  $R^W\in \Lambda^2(\Cl(\A^\ast))$. In order to show that $F^{W/S} \in \Lambda^2(\End_{\Cl(\A^\ast)} W)$ we will prove that the exterior multiplication $\epsilon(F^{W/S})$ 
acting on $\Gamma(W)$ commutes with Clifford multiplication $c(a)$ by an element  $a \in \A^{\ast}$.

Since $\nabla^W$ is admissible we have $[\nabla^W, c(a)] = c(\nabla a)$ where $\nabla$ is the connection obtained
from the fixed compatible metric $g$. 
Hence we get
\[
[(\nabla^{W})^2, c(a)] = [\nabla^W, [\nabla^W, c(a)]] = [\nabla^W, c(\nabla a)] = c(\nabla^2 a) = c(R a).
\]

We will next show that $R^W$ also satisfies the commutator property $[R^W, c(a)] = c(Ra)$. 
For then
\[
[F^{W/S}, c(a)] = [(\nabla^{W})^2, c(a)]  - [R^W, c(a)] = 0
\]
and hence $F^{W/S}$ is  an element of $\Lambda^2(\End_{\Cl(\A^\ast)} W)$. 

We 
identify $\A \cong \A^{\ast}$ via $g$ and write $a = \sum_{k=1}^n e^k(a) e_k$. Then
\begin{eqnarray}\label{R}
R(e_i, e_j) a = \sum_{l=1}^n g(R(e_i, e_j) a, e_l) e^l = \sum_{k,l=1}^n g(R(e_i, e_j) e_k, e_l) e^k(a) e^l.
\end{eqnarray}
It is sufficient to check the commutator property for $a=e^s$, $s=1,\ldots, n$. 
According to \eqref{Thm:Lichnerowicz.3} we then obtain
\begin{align*}
&R^W(e_i, e_j) c(e^s) - c(e^s) R^W(e_i, e_j) = \frac{1}{4} \sum_{k=1}^n g(R(e_i, e_j) e_k, e_s) c(e^k) c(e^s) c(e^s) \\
&-c(e^s) \frac{1}{4} \sum_{k=1}^n g(R(e_i, e_j) e_k, e_s) c(e^k) c(e^s) + \frac{1}{4} \sum_{l=1}^n g(R(e_i, e_j) e_s, e_l) c(e^s) c(e^l) c(e^s) \\
&+ c(e^s) \frac{1}{4} \sum_{l=1}^n g(R(e_i, e_j) e_s, e_l) c(e^s) c(e^l). 
\end{align*}

By Clifford multiplication all four terms take the form $\frac{1}{4} \sum_{l=1}^n g(R(e_i, e_j) e_s, e_l) c(e^l)$.
Together with \eqref{R} we find that  $[R^W(e_i, e_j), c(e^s)] = c(R(e_i,e_j) e^s)$, and this proves the claim.
\end {proof}

\begin{Thm}[Lichnerowicz formula]
Let $(M, \V, \A)$ be a Lie manifold, $S \to M$ a spin structure and $g = g_{\A}$
a compatible Riemannian metric. 
Denote by $\Cl(\A) \to M$ the Clifford bundle and let $W \in \Cl(\A) - mod$ be a Clifford module. 
Let $\nabla^{W}$ be an admissible connection and $D$ the corresponding geometric Dirac operator.
Then we have the formula
\[
D^2 = \Delta^W + c(F^{W / S}) + \frac{\kappa}{4},
\]

where $\kappa$ is the scalar curvature and  $F^{W / S} \in \Lambda^2(\End_{\Cl(\A)} W)$ is the twisting curvature.
\label{Thm:Lichnerowicz}
\end{Thm}

\begin{proof}
Let $R$ be the Riemannian curvature tensor induced by the fixed compatible metric $g$. 

If $c$ denotes the quantization map $\Lambda \to \Cl$, then $F^{W / S} \in \Lambda^2(\End_{\Cl(\A^\ast)}(W))$ has the image under $c$ 
\[
c(F^{W / S}) = \sum_{i < j} F^{W / S}(e_i, e_j) c(e^i) c(e^j).
\]

The scalar curvature $\kappa$ is given by 
\[
\kappa = \sum_{ik} R_{ikik}, \ R_{ijkl} := g(R(e_k, e_l) e_j, e_i). 
\]

Write $D = \sum_i c(e^i) \nabla_i^W$ for $\nabla_i^W$ the covariant derivative in direction $e_i$. 
This gives
\begin{align*}
D^2 &= \frac{1}{2} \sum_{i,j} [c(e^i), c(e^j)] \nabla_i^W \nabla_j^W + \sum_{i,j} c(e^i) [\nabla_i^W, c(e^j)] \nabla_j^W \\
&+ \frac{1}{2} \sum_{i,j} c(e^i) c(e^j) [\nabla_i^W, \nabla_j^W]. 
\end{align*}

By Clifford multiplication we have $[c(e^i), c(e^j)] = - 2 g^{ij}$, hence the first formula becomes $-\sum_{i,j} g^{ij} \nabla_i^W \nabla_j^W$.
Secondly, by admissibility of $\nabla^W$ it follows $[\nabla_i^W, c(e^j)] = c(\nabla_i e^j)$. 
Write $\nabla_i e^j = - \sum_k \Gamma_{ik}^j e^k$ in terms of Christoffel symbols. 
Then $[\nabla_i^W, c(e^j)] = - \sum_k \Gamma_{ik}^j c(e^k)$. 
Using the symmetry of $\Gamma_{ik}^j$ in $i$ and $k$ rewrite the second term
\begin{align*}
& \sum_{ij} c(e^i) [\nabla_i^W, c(e^j)] \nabla_j^W = \frac{1}{2} \sum_{i,k} [c(e^i), c(e^k)] \sum_{k} \Gamma_{ik}^j \nabla_j^W \\
&= - \sum_{i,k} g^{ik} \sum_{k} \Gamma_{ik}^j \nabla_j^W. 
\end{align*}

For the third term consider the curvature tensor $(\nabla^{W})^2$ and use $[e_i, e_j] = 0, \ i \not= j$ to obtain
\[
[\nabla_i^W, \nabla_j^W] = (\nabla^{W})^2(e_i, e_j). 
\]

Putting everything together $D^2$ is rewritten as 
\[
D^2 = - \sum_{i,j} g^{ij} (\nabla_i^W \nabla_j^W - \sum_k \Gamma_{ij}^k \nabla_k^W) + \frac{1}{2} \sum_{i,j} c(e^i) c(e^j) (\nabla^{W})^2(e_i, e_j). 
\]

Notice that the first term on the right is $\Delta^W$. 
Using Lemma \ref{Lem:twisting}, the second term is rewritten 
\[
\frac{1}{2} \sum_{i,j} c(e^i) c(e^j) (\nabla^{W})^2(e_i, e_j) = -\frac{1}{8} \sum_{ijkl} R_{ijkl} c(e^i) c(e^j) c(e^k) c(e^l) + c(F^{W / S}). 
\]

We next recall the identity  
\begin{eqnarray}\label{Thm:Lichnerowicz.4}
\lefteqn{c(e^i) c(e^j) c(e^k)}\nonumber\\
&=&
\frac{1}{3!} \sum_{\sigma \in S_3} \mathrm{sgn}(\sigma) c(e^{\sigma(i)}) c(e^{\sigma(j)}) c(e^{\sigma(k)}) - \delta^{ij} c(e^k) - \delta^{jk} c(e^i) + \delta^{ki} c(e^j). 
\end{eqnarray}

Rewrite $c(e^i) c(e^j) c(e^k)$ as in \eqref{Thm:Lichnerowicz.4}, 
recall the Bianchi identity $R_{ijkl} + R_{kijl} + R_{jkil} = 0$ and apply it together with \eqref{Thm:Lichnerowicz.4} to obtain
\begin{align*}
\sum_{ijkl} R_{ijkl} c(e^i) c(e^j) c(e^k)c(e^l) &= - \sum_{ijkl} R_{ijkl} (- \delta^{ij} c(e^k) - \delta^{jk} c(e^i) + \delta^{ki} c(e^j)) c(e^l) \\
&= - \sum_{ilk} R_{iikl} c(e^k) c(e^l) - \sum_{ikl} R_{ikkl} c(e^i) c(e^l) + \sum_{jkl} R_{kjkl} c(e^j) c(e^l). 
\end{align*}
Since $R$ is antisymmetric in the first two entries, the first term on the right hand side vanishes. Renaming indices
we obtain
\[
\sum_{ijkl} R_{ijkl} c(e^i) c(e^j) c(e^k) c(e^l) = 2 \sum_{ijk} R_{jkik} c(e^j) c(e^i). 
\]

Since $\sum_{ij} R_{jkik} c(e^j) c(e^i) = - \sum_{i} R_{ikik}$ we obtain the result.
\end{proof}

\section{Heat kernel approximation for Lie groupoids}



The heat kernel of a groupoid Laplacian is a  convolution kernel which has the properties expected of the heat
kernel. We recall the approximation of the heat kernel on Riemannian manifolds from Berline, Getzler and Vergne, \cite{BGV} 
and the corresponding approximation on Lie groupoids. 
Note that if $\G \rightrightarrows M$ is a Lie groupoid over a Lie manifold $(M, \A, \V)$ such that $\A(\G) \cong \A$,  
then an admissible metric $g$ yields a $C^{\infty}$-family of Riemannian metrics $(g_x)_{x \in M}$ on the $s$-fibers $(\G_x)_{x \in M}$.
By the definition of  submersions on manifolds with corners, the $s$-fibers are smooth manifolds without corners, cf. \cite{LN}. 
Additionally, $(\G_x, g_x)$ is a Riemannian manifold with uniformly bounded geometry and we refer to \cite{So2} for a proof of this. 

The class of Lie manifolds $(M, \A, \V)$ we consider in this section will be non-degenerate Lie structures whose  integrating groupoid is Hausdorff. 

We will give examples of such Lie structures below.

\subsection*{The Schwartz class.}
Let us fix for the moment a Lie groupoid $\G \rightrightarrows M$ and a Haar system $\{\mu_x\}_{x \in M}$ on $\G$ such that there is a \emph{length function}, i.e.
a function $\varphi \colon \G \to \overline\Rr_{+}$ which has the following properties:
\begin{enumerate}\renewcommand{\labelenumi}{(\roman{enumi})}
\item  $\varphi(\gamma_1 \gamma_2) \leq \varphi(\gamma_1) + \varphi(\gamma_2)$ for $(\gamma_1, \gamma_2) \in \Gpull$. 

\item $\varphi(\gamma^{-1}) = \varphi(\gamma)^{-1}, \ \gamma \in \G$. 

\item $\varphi$ is proper.

In \cite{LMN} the authors require in addition: 

\item  $\varphi$ is of polynomial growth, i.e. there is a $C > 0$ and $N \in \Nn$ such that for each $r \in \Rr_{+}$
we have $\mu_x(\varphi^{-1}([0,r])) \leq C (r^N + 1)$. 
\end{enumerate}
In the sequel, we will not make use of assumption (iv). It guarantees that, for $k$ sufficiently large, the integral $\int_{\G_x} \frac{1}{(1 + \varphi(\gamma))^k} \,d\mu_x(\gamma)$ remains
uniformly bounded. We will recall below some of the consequences of this additional property. 

A vector field $v$ in $\Gamma(\A(\G)) = \V$ can be regarded as a $\G$-invariant first order differential operator on $\G$ (by lifting $v$ to the $s$-vertical tangent bundle of the groupoid). 
We denote by $(v_1, \cdots, v_l) \mapsto \omega_{\overline{v},i}$ the distributional action $\omega_{\overline{v},i}(f) = v_1 \cdots v_i f v_{i+1} \cdots v_l$ for $f \in C_0(\G)$, the space of continuous functions on $\G$ vanishing at infinity,  considered as a convolution operator $C_c^{0,\infty}(\G)\to C^{0,\infty}(\G)$. 
Define
\[
S_{\varphi}^{k,0}(\G) := \{f \in C_0(\G) : \sup_{\gamma \in \G} |\Pp(\varphi(\gamma)) f(\gamma)| < \infty, \ \Pp \in \Rr[X], \ \mathrm{deg}(\Pp) = k\}. 
\]

Also define the spaces
\[
S_{\varphi}^{k,l}(\G) := \{f \in C_{0}(\G) : \|f\|_{\Pp, l} < \infty, \ \Pp \in \Rr[X], \ \deg(\Pp) = k\}.
\]

Here we denote by $\|\cdot\|_{\Pp, l}$ for  $l \in \Nn$ and a given polynomial $\Pp \in \Rr[X]$ of degree $k$, the seminorms
\[
\|f\|_{\Pp, l} := \sup_{1 \leq i \leq l} \ \sup_{\|v_j\| \leq 1, \ \overline{v} = (v_1, \cdots, v_l) \in \V} \ \sup_{\gamma \in \G} |\Pp(\varphi(\gamma)) \omega_{\overline{v}, i}(f)|.
\]

\begin{Prop}
The spaces $\{S_{\varphi}^{k,l}(\G)\}_{k,l \in \Nn}$ form a dense projective system of Banach spaces. 
\label{Prop:proj}
\end{Prop}

\begin{proof}
Apart from  the semi-norm system $\|\cdot\|_{\Pp,l}$ 
parametrized by $\Pp \in \Rr[X]$ and $l \in \Nn$
we  have the equivalent system $\{\|\cdot\|_{k,l}\}_{k,l \in \Nn}$ where 
\[
f \mapsto \|f\|_{k,l} := \sup_{1 \leq i \leq l} \sup_{\|v_j\| \leq 1, \ \overline{v} = (v_1, \cdots, v_l)} \sup_{\gamma \in \G} (1 + \varphi(\gamma))^k |\omega_{\overline{v}, i}(f)|. 
\] 

For the projectivity we observe that if $l$ is fixed and $k_1 \geq k_2$ then $\|\cdot\|_{k_1, l} \leq \|\cdot\|_{k_2, l}$.
Secondly, if $k$ is fixed and $l_1 \geq l_2$ then $\|\cdot\|_{k, l_1} \leq \|\cdot\|_{k, l_2}$. The density of the inclusions
is immediate. 
Hence $\{S_{\varphi}^{l,k}(\G)\}_{(l,k) \in \Nn^2}$ forms a dense projective system of Banach spaces. 
\end{proof}

\begin{Def}
Let $\G \rightrightarrows M$ be a Lie groupoid with length function $\varphi \colon \G \to \Rr_{+}$. We
define the space of rapidly decaying distributions as the dense projective limit
\[
\S_{\varphi}(\G) := \varprojlim_{k,l \in \Nn} S_{\varphi}^{k,l}(\G). 
\]
\label{Def:Schwartz}
\end{Def}

If the length function is of polynomial growth the class is closed under holomorphic functional calculus, see \cite[Theorem 7.5]{LMN}. 

\begin{Prop}
Let $\G \rightrightarrows M$ be a Lie groupoid with polynomial length function $\varphi$. Then $\S_{\varphi}(\G)$ is a $\ast$-subalgebra of $C_{r}^{\ast}(\G)$,
stable under holomorphic functional calculus.
\label{Prop:Schwartz}
\end{Prop}

In fact, it is shown in \cite[Lemma 7.8]{LMN} that $S_\varphi^{k,l}(\G)$ is closed under holomorphic functional calculus in $C^\ast_r(\G)$ for large  $k$, hence so is $\S_\varphi(\G)$.
 
\begin{Exa}
\emph{i)} Let $M$ be a compact manifold with embedded corners and 
$\{\rho_i\}_{i=1}^N$ a set of boundary defining functions.
The boundary of $M$ is stratified by the closed, codimension one hyperfaces $F_i = \{\rho_i = 0\}$: 
 $$\partial M = \bigcup_{1 \leq i \leq N} F_i.$$ 
We consider the Lie structure $\V_b := \{V \in \Gamma^{\infty}(TM) : V \ \text{tangent to} \ F_i, \ 1 \leq i \leq N\}$. 
The Lie algebroid $\A \to M$ is the $b$-tangent bundle such that $\Gamma(\A) \cong \V_b$. 
Following Monthubert \cite{Mont2}, we find a Lie groupoid $\G_b(M)$ integrating $\A$ which is $s$-connected, Hausdorff and amenable: 
We start with the set 
\[
\Gamma_b(M) = \{(x, y, \lambda) \in M \times M \times (\Rr_{+})^N : \rho_i(x) = \lambda_i \rho_i(y), \ 1 \leq i \leq N\}
\]
endowed with the  structure  $(x, y, \lambda) \circ (y, z, \mu) = (x, z, \lambda \cdot \mu), \ (x, y, \lambda)^{-1} = (y,x, \lambda^{-1})$
and $r(x, y, \lambda) = x, \ s(x, y, \lambda) = y,\ u(x) = (x,x,1)$. Here multiplication $\lambda\cdot \mu$ and inversion $\lambda^{-1}$ are componentwise.

We then define the $b$-groupoid $\G_b(M)$ as the $s$-connected component (the union of the connected components of the $s$-fibers of $\Gamma_b(M)$),
i.e. $\G_b(M) := \C_s \Gamma_b(M)$. 
The $b$-groupoid has the polynomial length function $\varphi(x, y, \lambda) = |\ln(\lambda)|$, cf. \cite{LMN}. 

\emph{ii)} Let $M$ be a compact manifold with corners as in the previous example. Fix the Lie structure $\V_{c_l}$ of 
generalized cusp vector fields for $l \geq 2$ given by the local generators in a tubular neighborhood of a boundary hyperface:
$\{x_1^l \partial_{x_1}, \partial_{x_2}, \cdots, \partial_{x_n}\}$. 
Let us recall the construction of the associated Lie groupoid $\G_l(M)$, the so-called generalized cusp groupoid,  given in \cite{LMN}. 
We set
\[
\Gamma_l(M) := \{(x, y, \mu) \in M \times M \times (\Rr_{+})^N : \mu_i \rho_i(x)^l \rho_i(y)^l = \rho_i(x)^l - \rho_i(y)^l\}
\]
with structure $r(x, y, \lambda) = x, \ s(x, y, \lambda) = y, \ u(x) = (x,x,0)$ and $(x, y, \lambda) (y, z, \mu) = (x, z, \lambda + \mu)$. We then define $\G_l(M)$ as the $s$-connected component of $\Gamma_l(M)$.
There exists a homeomorphism $\Theta_l \colon \G_b(M) \to \G_l(M)$ given by $(x, y, \lambda) \mapsto (u, v, \mu)$
as follows.
Assume first that $M$ has only one boundary hyperface, i.e.\ $M$ is a manifold with boundary. The generalization to arbitrarily
many hyperfaces is easy.
We then partition $M$ into $M = \U \cup (M \setminus \U)$ where $\U$ is a collar neighborhood of the boundary.
Then
\[
u = \begin{cases} x, \ x \in M \setminus \U \\
\pi^{-1} \circ \tau_l \circ \pi(x), \ x \in \U, \end{cases}
\]
where $\pi:\U\cong \partial M \times [0,1)$, and
\[
v = \begin{cases} y, \ y \in M \setminus \U \\
\pi^{-1} \circ \tau_l \circ \pi(y), \ y \in \U \end{cases}
\]

Here $\tau_l \colon \Rr_{+} \to \Rr_{+}$ is the continuous and strictly increasing function given by
\[
t \mapsto \begin{cases} 0, \ t = 0, \\\frac{1}{e} (-\ln(t))^{-\frac{1}{l}}, \ t \in \left(0, \frac{1}{e}\right), \\
t, \ t \geq \frac1e \end{cases}.
\]
Set $\mu = \log(\lambda)$ and check with the above that $\mu \rho(u)^l \rho(v)^l = \rho(u)^l - \rho(v)^l$.
This transformation in a tubular neighborhood of the boundary motivates the definition of the cusp groupoid.
The polynomial length function on the cusp groupoid is then obtained by $\varphi_l := \varphi \circ \Theta_l^{-1}$ where $\varphi$
denotes the polynomial length function of the $b$-groupoid and $\Theta_l$ is the homeomorphism constructed above.
We obtain that $\varphi(x, y, \mu) = |\mu|$. 

\emph{iii)} The following example of the fibered cusp calculus is from Mazzeo and Melrose \cite{MM} and we use the formulation and notation for manifolds with iterated fibered corners as given in \cite{DLR}. 
We briefly recall the definition of the associated groupoid and refer to loc. cit. for the details.  
Let $M$ be a manifold with embedded and iterated fibered corners. 
We denote by $\{F_i\}_{i=1}^N$ the boundary hyperfaces of $M$ with boundary defining functions $\rho_i$ and write 
$\pi = (\pi_1, \cdots, \pi_N)$, where $\pi_i \colon F_i \to B_i$ are fibrations; 
$B_i$ is the base, which is 
a compact manifold with corners. 
Define the Lie structure
\[
\V_{\pi} := \{V \in \V_b : V_{|F_i} \ \text{tangent to the fibers} \ \pi_i \colon F_i \to B_i, \ V \rho_i \in \rho_i^2 C^{\infty}(M)\}.
\]
Then $\V_{\pi}$ is a finitely generated $C^{\infty}(M)$-module and a Lie sub-algebra of $\Gamma^{\infty}(TM)$. 
The corresponding groupoid is amenable  \cite[Lemma 4.6]{DLR} and by construction as an open submanifold of a compact Hausdorff space it is Hausdorff; as a set it is defined as
\[
\G_{\pi}(M) := (M_0 \times M_0) \cup \left(\bigcup_{i = 1}^{N} (F_i \times_{\pi_i} T^{\pi} B_i \times_{\pi_i} F_i) \times \Rr\right), 
\]
where $T^{\pi} B_i$ denotes the algebroid of $B_i$.

A different type of manifold with fibered boundary has been considered in \cite{Guil} as well as the corresponding Lie groupoid with length function.

\label{Exa:cuspidal}
\end{Exa}

\subsection*{Heat kernel approximation}

In view of  the right invariance of the action of $\G$ on itself we consider the family of metrics $(g_x)_{x \in M}$ on the $s$-fibers
of the groupoid. We denote the family of induced metric distances by $(\dd_x)_{x \in M}$ and note that
this is a $\G$-invariant family as well, i.e.,  
\[
\dd_{s(\gamma)}(\gamma_1 \gamma, \gamma_2 \gamma) = 
\dd_{r(\gamma)}(\gamma_1, \gamma_2).
\]

Given $\gamma, \eta \in \G_{s(\gamma)}$ we see from this that 
$\dd_{s(\gamma)}(\gamma, \eta) = 
\dd_{r(\gamma)}(\id_{r(\gamma)}, \eta \gamma^{-1})$.
Hence we can define a \emph{reduced metric distance} by $\psi(\gamma) := \dd_{s(\gamma)}(\id_{s(\gamma)}, \gamma)$. 

\begin{Lem}
The reduced metric distance $\psi(\gamma) = \dd_{s(\gamma)}(\id_{s(\gamma)}, \gamma)$ is a length function, i.e. if
$(\gamma, \eta) \in \Gpull$ then $\psi(\gamma \eta) \leq \psi(\gamma) + \psi(\eta)$ and for each $\gamma \in \G$
we have $\psi(\gamma^{-1}) = \psi(\gamma)$. Moreover, it is proper.
\label{Lem:redmetric}
\end{Lem}

\begin{proof}
First apply the triangle inequality, the $\G$-invariance and the fact that $r(\eta) = s(\gamma)$ by composability to obtain
\begin{align*}
\psi(\gamma \eta) &= \dd_{s(\gamma \eta)}(\id_{s(\gamma \eta)}, \gamma \eta) \\
&\leq \dd_{s(\gamma \eta)}(\id_{s(\gamma \eta)}, \eta) + 
\dd_{s(\gamma \eta)}(\eta, \gamma \eta) \\
&= \dd_{s(\eta)}(\id_{s(\eta)}, \eta) + \dd_{s(\gamma)}(\id_{s(\gamma)}, \gamma). 
\end{align*}

Secondly, by right invariance 
\begin{align*}
\psi(\gamma^{-1}) &= \dd_{s(\gamma^{-1})}(\id_{s(\gamma^{-1})}, \gamma^{-1}) = \dd_{r(\gamma^{-1})}(\gamma, \id_{r(\gamma^{-1})}) \\
&= \dd_{s(\gamma)}(\gamma, \id_{s(\gamma)}) = \psi(\gamma).
\end{align*}
The fibers of the Lie groupoid are complete manifolds. Hence the theorem of Hopf-Rinow implies that the preimages of bounded sets are bounded. Moreover, a complete Riemannian manifold is a Montel space, i.e. bounded and closed subsets are compact. Therefore,   $\psi$ is proper.  
\end{proof}

\begin{Def}
Given a Hausdorff Lie groupoid $\G \rightrightarrows \Gop$, we write $\S(\G) := \S_{\psi}(\G)$, where $\psi$ is
the reduced metric distance of $\G$.
\label{Def:polygrowth}
\end{Def}

On the Lie groupoid $\G \rightrightarrows \Gop = M$ we introduce the heat kernel for the generalized (twisted) Laplacian $\Dirac^2$ depending on the admissible connection $\nablaslash^W$ (see Section \ref{section2} for the definitions and notation).
If $g = g_{\A}$ is a compatible metric induced on $(M, \A, \V)$, then $(M_0, g)$ is also a manifold with bounded geometry, see \cite{ALN2}.
Let us fix an invariant connection $\nabla$ on $\G$ which is obtained from the $\G$-invariant family of connections $(\nabla_x)_{x \in \Gop}$ associated to the
$\G$-invariant family of metrics $(g_x)_{x \in \Gop}$. 
Varying $x \in \Gop$ the family of exponential mappings $\exp_x \colon T \G_x \to \G_x$ yields an exponential
mapping $\Exp \colon \A \to \G$, see \cite[p. 128f]{NWX}. For more details see Section \ref{section:IV}.

Let $r_0>0$ be the bounded injectivity radius. Then the induced exponential mapping $\Exp$ maps $(\A)_{r_0} := \{v \in \A : \|v\|_{g} < r_0\}$
diffeomorphically onto its image $\B_{r_0} := \{\gamma \in \G : \dd_{s(\gamma)}(\gamma, s(\gamma)) < r_0\}$. 
We fix polar coordinates  $(p, \theta)$ on $\A_x$ such that $\dd_x(\Exp(p, \theta), x) = p$. 
Define the radial vector field $\partial_{\R} := \dd_{s(\gamma)}(\gamma, s(\gamma)) d \,\Exp(\partial_p)$, $s(\gamma) = x$, 
and set
$J := \det(d \,\Exp) \circ (\Exp)^{-1}$. 

Consider the pullbacks  $r^{\ast} W \to \G$ and $s^{\ast} W \to \G$ of the Clifford module $W\to M$.
We denote by $\tau(\gamma)(w) \in r^{\ast} W_{\gamma}$ the parallel transport of $w \in W_{s(\gamma)}$ along the path $\Exp(t v), \ t \in [0,1]$, where $\gamma = \Exp\, v \in \B_{r_0}$, for $v \in \A_{s(\gamma)}$. This defines  a map
\[
\tau \colon \{(\gamma, w) : \gamma \in \B_{r_0}, \ w \in W_{s(\gamma)}\} \to r^{\ast} W_{|\B_{r_0}}.
\]
The inverse is given by $\tau(\gamma)^{-1} \colon r^{\ast} W_{|\G_x \cap \B_{r_0}} \to W_x$. 

Denote by $r^{\ast} W \otimes s^{\ast} W^{\ast} \times (0, \infty)$ the pullback of the vector bundle $r^{\ast} W \otimes s^{\ast} W^{\ast} \to \G$
along the projection $\G \times (0,\infty) \to \G$. 

The groupoid heat kernel is a $C^0$-section $Q \in \Gamma^0(r^{\ast} W \otimes s^{\ast} W^{\ast} \times (0,\infty))$ such that  for $Q_t = Q(t,\cdot)$
\begin{enumerate}\renewcommand{\labelenumi}{(\roman{enumi})}
\item the heat equation $(\partial_t + \Dirac^2) Q_t(\gamma) = 0$ holds with
\item the initial condition $\lim_{t \to 0} Q_t \ast u = u$   for each $u \in \Gamma_c^{\infty}(r^{\ast} W \otimes s^{\ast} W^{\ast})$. 
\end{enumerate}
Since the generalized Laplacian $\Dirac^2$ on $\G$ comes from an equivariant family (compare the remarks in Section \ref{section2}) 
the map  $\G_x \times \G_x \ni (\gamma, \eta) \mapsto Q_t(\gamma \eta^{-1})$ defines a heat kernel for $\Dirac_x^2$ on $\G_x$ for each $x \in M$. Since $\G_x$ has bounded geometry  the heat kernel of $\Dirac_x^2$ is unique (cf. \cite[Proposition 2.17]{BGV}).
Hence, by $\G$-invariance,  $Q$ must be unique as well.

We repeat the formal heat kernel approximation from \cite[Section 2.5]{BGV} and, more specifically, from \cite{So}. 

Let $q \colon \B_{r_0} \times (0, \infty) \to \Rr$ denote the Gaussian
\[
q(\gamma, t) := (4 \pi t)^{-\frac{n}{2}} e^{-\frac{\dd_{s(\gamma)}(\gamma, s(\gamma))^2}{4t}}.
\] 

Following \cite{BGV} and \cite{So2} we start with an Ansatz for a formal solution in the form $q\Phi$, where $\Phi$  is a formal power series
\begin{align}
& \Phi(t,\gamma) = \sum_{i=0}^{\infty} t^i \Phi_i(\gamma), \ \Phi_i \in \Gamma^\infty (r^{\ast} W \otimes s^{\ast} W^{\ast}). \label{Phi}
\end{align}
Then
\begin{align}
& (\partial_t + \Dirac^2) (q(\gamma, t) \Phi(t,\gamma) )
= q(t,\gamma) \left(\partial_t + \Dirac^2 + t^{-1} \nabla_{\partial_{\R}} + \frac{\mathcal{L}_{\partial_{\R}} J}{2 t J}\right) \Phi(t,\gamma). \label{happrox}
\end{align}

The assumption that $(\partial_t +\Dirac^2)(q\Phi) =0$ leads to the recursive system of equations
\begin{align*}
& \Phi_0(\Exp \,V) = J^{-\frac{1}{2}} \tau(\Exp \,V), \\
& \Phi_i(\Exp\, V) = -J^{-\frac{1}{2}} \tau \int_0^1 J^{\frac{1}{2}} \tau^{-1}((\Dirac^2 \Phi_{i-1})(\Exp(tV))) t^{i-1}\,dt.
\end{align*}
Then $Q_t$ is obtained as a series
$$Q_t(\gamma) = 
\sum_{k=0}^\infty (-1)^kQ_N^{(k)}(t, \gamma),\quad \gamma \in \G,$$
for sufficiently large $N$, where 
\begin{eqnarray*}
Q_N^{(0)}&=& G_N\\
Q_N^{(k)}(t, \gamma)&=& \int_0^t G_N(t-s,\gamma) R^{(k)}_N(s,\gamma) ds, \quad k\ge 1, \\
G_N(t,\gamma) &=& q(t,\gamma)\chi(\gamma) \sum_{i=0}^Nt^i\Phi_i(\gamma)
\end{eqnarray*} 
for  a cut-off function $\chi$, supported in $\B_{r_0}$ and equal to $1$ on $\B_{r_0/2}$. Furthermore,
\begin{eqnarray*}
R_N^{(1)} &=&(\partial_t+\Dirac^2) G_N \text{ and}\\
R_N^{(k)} (t,\gamma)&=& \int_0^t R_N^{(1)}(t-s,\gamma)R^{(k-1)}_N (s,\gamma)\,ds\\
&=&\int_0^t\int_{\G_{s(\gamma)}}R_N^{(1)}(t-s,\gamma\eta^{-1})R^{(k-1)}_N (s,\eta)\,d\mu_{s(\gamma)}(\eta) ds,\quad k\ge2.
\end{eqnarray*}

\begin{Thm}
Let $\G \rightrightarrows \Gop$ be a Hausdorff Lie groupoid,
and let $\Dirac$ be a geometric Dirac operator adapted to an admissible connection.
Then for $t > 0$ we have $e^{-t \Dirac^2} \in \S(\G)$.  
\label{Thm:rapdecay}
\end{Thm}

\begin{proof}
We use the heat kernel approximation and the bounded geometry of the groupoid fibers.
Since the smooth sections of $\A(\G)$ are in one-to-one correspondence with $\G$-invariant vector fields on $T^s\G$, see \cite[p.122]{NWX}, we deduce from \cite[Lemma 4.8]{So2} the following uniform estimate:
\begin{align}
& |V_1\cdots V_i Q_N^{(k)} (t,\cdot) V_{i+1}\cdots V_l(\gamma)|
\le C_1C_2^k \left(1+t^{N-(n+l)/2}\right)^k\frac{t^k}{k!} \label{hineq}
\end{align}
for fixed $t>0$ and $V_1, \ldots ,V_l\in \V$, $1\le i\le l$.
This follows from corresponding uniform estimates for $G_N$ and $R_N^{(k)}$ which 
hold in view of the boundedness of the $\Phi_i$ and the bounded geometry.

Next we  will prove the following off-diagonal estimate of the heat kernel:
Let $V_1, \ldots ,V_l\in \V$, $1\le i\le l$ be given, and $t > 0$ fixed. 
Then for each $\lambda>0$ there exists a $C>0$ such that
\begin{align}
|V_1\cdots V_i Q_t V_{i+1}\cdots V_l(\gamma)| \leq C e^{-\lambda \psi(\gamma)}, \ \gamma \in \G, \ \psi(\gamma) > 2 r_0. \label{estimate}
\end{align}

In fact, the support condition on $\chi$ implies that $G_N(t, \gamma \eta^{-1}) = 0$ for  $\eta \in \G_{s(\gamma)} \setminus \B_{r_0}$. 
Let $I \in \Nn$ such that  $I r_0 <  \dd(\gamma, s(\gamma)) \le (I + 1) r_0$. Then  $Q_N^{(k)}(t, \gamma) = 0$ for $k < I$:
This follows from the fact that $R_N^{(k)}$ can be written in the form 
\begin{eqnarray*}
\lefteqn{R_N^{(k)}(t, \gamma)  }\\
&=& \int_{\Delta_k} \int_{\eta_1, \cdots, \eta_{k-1} \in \G_{s(\gamma)}} R_N^{(1)}(t - t_1, \gamma \eta_1^{-1}) R_N^{(1)}(t_1 - t_2, \eta_1 \eta_2^{-1}) \cdots \\
&&R_N^{(1)}(t_{k-2} - t_{k-1}, \eta_{k-1} \eta_k^{-1}) R_N^{(1)}(\eta_{k-1}, t_{k-1}) \,d\mu_{s(\gamma)}(\eta_1) \cdots \,d\mu_{s(\gamma)}(\eta_{k-1})
\end{eqnarray*}
where $\Delta_k := \{(t_1, \cdots, t_{k-1}) : 0 \leq t_{k-1} \leq \cdots \leq t_1 \leq t\}$. 
There are $k$ factors  in the integral on the right hand side. An application of the iterated triangle inequality:
\[
k r_0 < I r_0 <  \dd_{s(\gamma)}(\gamma, s(\gamma)) \le \dd_{s(\gamma)}(\gamma, \eta_1) + \sum_{j=2}^{k-1} \dd_{s(\gamma)}(\eta_{j-1}, \eta_j)+\dd_{s(\gamma)}(\eta_{k-1},s(\gamma))
\]
implies that at least one of them is $\geq r_0$. Hence $R_N^{(k)}(t, \gamma) = 0$  and therefore $Q_N^{(k)}(t, \gamma) = 0$. 
This fact together with the estimate \eqref{hineq} yields
\begin{align}
&|V_1 \cdots V_i Q_t V_{i+1} \cdots V_l(\gamma)| e^{\lambda \psi(\gamma)} \leq \sum_{k=I}^{\infty} e^{\lambda (I + 1) r_0} C_1 C_2^k (1 + t^{N - \frac{n+l}{2}})^k \frac{t^k}{k!} \nonumber\\
&= e^{\lambda(I + 1) r_0} \frac{C_1 C_2^I (1 + t^{N- \frac{n+l}{2}})^I t^I}{I!} \sum_{k=0}^{\infty} \frac{C_2^k (1 + t^{N - \frac{n+l}{2}})^k t^k I!}{(k+I)!}\label{rapdecay2}
\end{align}

and the right hand side tends to $0$ as $I \to \infty$. 

Now \eqref{hineq} and \eqref{rapdecay2} together with the fact that $e^{-\lambda \psi(\gamma)}$ decays faster than any polynomial in the reduced length function $\psi(\gamma)$ imply that we can estimate the  $\S(\G)$-seminorms of the groupoid heat kernel $Q_t$
\begin{align*}
\|Q_t\|_{k,l} 
&=  \sup_{1 \leq i \leq l} \sup_{\|V_j\| \leq 1} \sup_{\gamma \in \G} (1 + \dd_{s(\gamma)}(s(\gamma), \gamma))^k |V_1 \cdots V_i 
Q_t V_{i+1} \cdots V_l (\gamma) | <\infty.
\end{align*}
%
\end{proof}

\section{Adiabatic deformation and functional calculus}

\label{section:IV}

Given a Lie groupoid $\G \rightrightarrows M$ over a smooth manifold $M$ with corners
we define the \emph{adiabatic groupoid} $\Gad \rightrightarrows M \times I$, where $I$ is either the closed interval
$[0,1]$ or the real numbers $\Rr$. We also write $I^{\ast} := I \setminus \{0\}$.
Formally, the groupoid $\Gad$ is defined as $\Gad =  \A(\G) \times \{0\}\cup \G \times I^{\ast} $. 
The groupoid structure over $t \not= 0$ is defined to be the structure of $\G$ and $I^{\ast}$, where the latter
is simply viewed as the trivial set groupoid with units $I^{\ast}$. Over $t = 0$ the structure is given by that of $\A(\G)$,
where we view $\A(\G) = \bigcup_{x \in M} T_{u(x)} \G_x$ as a bundle with fiberwise defined Lie group structure.  

Note right away the most important special case of this definition: If $M$ is a smooth manifold without boundary or corners, and  $M \times M \rightrightarrows M$ is the pair groupoid,
we recover with $(M \times M)^{ad}$ the \emph{tangent groupoid} due to A. Connes, \cite{C}. 
In the more general situation we are in, where $M$ may have
a boundary or corners, we need an integrating groupoid $\G$ which is more general, in particular longitudinally smooth.

Most important for us is the smooth structure defined on $\Gad$, which we will need later. 
In the special case of the tangent groupoid one fixes a Riemannian metric on $M$ with Levi-Civita connection and defines the topology of $(M \times M)^{ad}$ via a glueing using the exponential mapping $\exp \colon TM \to M$, cf. \cite{C} for the tangent groupoid. 
In the more general case of the adiabatic groupoid at hand we need the so-called \emph{generalized exponential} 
$\Exp \colon \A(\G) \to \G$ from \cite{LR}. This definition together with the \emph{groupoid parametrization} defined in the same paper will be most
convenient for the calculation of the Lichnerowicz formula on the fibers of the adiabatic groupoid in the final section of this article.

\begin{Def}\label{lift} Write $T^s\G :=\ker ds \subseteq T\G$ and let $ \pi: T\G\to \G$ be the canonical projection.  
The right \emph{generalized exponential map} $\Exp^{R} \colon \A(\G) \to \G$ is defined by
\[
\Exp^{R}(V) := 
\pi_{\ker ds \to \G}(V(1))
\]
where $V(1)$ is the solution at $t=1$ to the flow equation $V'(t) = \l_{V(t)} V(t), \ V(0) = V \in \A(\G)$ (provided it exists) and $\l_X: T_{\pi X}\G \to T_XT^s\G$, 
$X\in T^s\G$, is the so-called horizontal lift, which will be defined below. 
\label{Def:Exp}

A left exponential map can be defined analogously with $\ker dr$ in place of $\ker ds$. 
\end{Def}

\begin{Rem}Writing  $r^{\ast} \A(\G) = \{(\gamma,v)\in \G\times \ker ds|_{\G^{(0)}}: 
\pi v = r(\gamma)\}$,  the map $(\gamma,v) \mapsto dR_{\gamma} v$ with the differential $dR_{\gamma}$ of the right multiplication $R_{\gamma}$ by $\gamma \in \G$ furnishes an isomorphism $r^\ast \A(\G)\to \ker ds $. 
In fact,  $dR_{\gamma}$ maps $(\ker ds)_{u(r(\gamma))} $ to $(\ker ds)_{\gamma}$, since 
$ds \circ d R_{\gamma} = d(s \circ R_{\gamma}) = d(s(\gamma)) = 0$. 

The connection $\nabla$ on $\A(\G)$ then lifts to a connection $\tilde \nabla$ on $T^s\G=r^\ast \A(\G)$. 
\end{Rem}

Given a smooth manifold $B$ and a vector bundle $\pi \colon E \to B$  with connection $\nabla^E$, we obtain a splitting $TE = T^{vert}E \oplus T^{hor}E$ of $TE$ with $T^{vert}E = \ker d\pi$. Associated with the decomposition  we have a lift of vectors: for $b\in B$ and $e\in E_b$ we have a lift $$l^E_e: T_bB\to T^{hor}_eE$$
via parallel transport.
We can also lift a curve $\gamma\colon [0,1]\to B$. Let  $\gamma(0)=b_0$ and $v_0\in E_{b_0}$. We obtain the lift $\Gamma:[0,1] \to E$ 
of $\gamma $ by solving the initial value problem  
$$\dot\Gamma(t) = l^E_{\Gamma(t)}(\dot\gamma(t)), \quad \Gamma(0) = v_0.$$   

In order to define the horizontal lift in Definition \ref{lift}, we 
recall that $\A(\G)$ is the restriction of $\ker ds$ to the units, 
i.e. $\A(\G) = u^{\ast} T^s \G$.
As observed in \cite{LR}, the above connection $\tilde \nabla$  defines parallel transport on $T^s\G$ in the same way as on a tangent bundle with affine connection. 
We hence obtain a lift  $l^{T^s\G}$ of curves in $\G$ to curves in $T^{s}\G$. 

The \emph{horizontal lift} $\l^{T^s\G}_X(V)$ 
of a tangent vector $V = \frac{d\gamma(t)}{dt}_{|t=0}$ in $T_{\gamma} \G$ to $X \in T_{\gamma}^s \G$ can be written explicitly in the form 
\[
\l^{T^s\G}_X(V) = \frac{d}{dt} \left[dR_{\gamma(t)} \l_{d R^{-1}_{\gamma}(X)}^{\A(\G)}(s(\gamma(t)))\right]_{t=0},
\]
where we denote by $\l^{\A(\G)}_{\bullet}$ the horizontal lift to the bundle $\A(\G) \to \Gop$ which lifts the curve  $t\mapsto s(\gamma(t))$, cf. \cite[(2.6)]{LR}.

\begin{Rem}The following diagram commutes
\[
\xymatrix{
\A(\G) \ar[rd]_{\pi} \ar[r]^{\Exp^R} & \G \ar[d]_{s} \\
& \Gop
}
\]
i.e. $s(\Exp^R(V)) = \pi(V)$ holds, cf. \cite[(2.9)]{LR}.
\end{Rem}


Finally, we recall the \emph{tubular neighborhood theorem} of the generalized exponential from \cite[Proposition 2.6]{LR}:
There are open neighborhoods $\Gop \subset U_1 \subset \A(\G)$ of the zero section in $\A(\G)$ 
and 
$\Gop \subset U_2 \subset \G$ of the unit space in $\G$ such that $\Exp^R(x) = x$ for each $x \in \Gop$ and $\Exp^R$ induces
a diffeomorphism of $U_1$ and $U_2$. 

From now on we simply write $\Exp := \Exp^R$, where it is understood that all our constructions are right invariant.
The smooth structure of the adiabatic groupoid $\Gad \rightrightarrows \Gop \times I$
is defined by glueing a neighborhood $\O$ of $\A(\G)\times \{0\}$ to $\G\times I^\ast$ via
\[
\O \ni (v, t) \mapsto \begin{cases} v, \ t = 0 \\
(\Exp(-tv), t), \ t > 0. \end{cases}
\]
\medskip

\subsection*{Functional calculus}

Let $(\M, g)$ denote a complete Riemannian manifold which is spin and $W \to \M$ a graded Clifford bundle.
The next goal is to define the functional calculus for rapidly decaying functions with values in the groupoid convolution
algebra.
Note that for
complete Riemannian manifolds there is a spectral theorem for such operators, see also \cite{Ch}, \cite{MS}. 

Given a Dirac operator $D$ acting on smooth sections of the graded Clifford bundle $W$ and $f \in \S(\Rr)$ we define the 
operator $f(D) = \frac{1}{2\pi} \int \hat{f}(t) e^{itD}\,dt$
in the weak sense, i.e. there is an $f(D)$ acting on $L^2(W)$ such that for each $s, \tilde{s} \in L^2$ we have by
\[
\scal{f(D) s}{\tilde{s}} = \frac{1}{2\pi} \int \hat{f}(t) \scal{e^{itD} s}{\tilde{s}}\,dt, \ s, \tilde{s} \in L^2(W). 
\]

Another notion we need to recall here is that of \emph{finite propagation speed}. 

If $D$ is an operator of first order on $\M$ and $\sigma_1$ its principal symbol, we denote by 
\[
\mathbf c(x) := \sup\{\|\sigma_1(x, \xi)\| : \|\xi\| = 1\}
\]

the \emph{propagation speed} of $D$.

\begin{Def}
A first order differential operator $D$ has \emph{finite propagation speed} if there is a constant $C > 0$ such that we have the 
uniform bound $\mathbf c(x) \leq C$. 
\label{Def:finspeed}
\end{Def}

We recall the following theorem due to Chernoff from \cite{Ch}.

\begin{Thm}[P. R. Chernoff, 1973]
Let $D \colon \Gamma(E) \to \Gamma(E)$ be a first order differential operator over a non-compact complete manifold
and assume that $D$ is formally self-adjoint and has finite propagation speed.
Then $D^k$  is essentially self-adjoint for $k \in \Nn_0$.
\label{Thm:Chernoff}
\end{Thm}


\begin{Prop}
Let $D$ be a Dirac operator acting on sections of the graded Clifford bundle $W \to \M$.

\emph{i)} The wave equation $\partial_t s = i D s$ with initial data $s_0 \in \Gamma_c^\infty(W)$ has a unique solution which
preserves the $L^2$-norm.

\emph{ii)} The operator $f(D)$ is well-defined and bounded on $L^2(W)$.

\emph{iii)} The assignment $\S(\Rr) \to \L(L^2(W))$ is a ring-homomorphism such that $\|f(D)\| \leq \|f\|_{\infty}$.

\emph{iv)} If $\hat f$ has compact support, then $f(D)$ is a smoothing operator with finite propagation speed, and $f(D)$ is essentially self-adjoint.

\label{Prop:fctcalc}
\end{Prop}

\begin{proof}
\emph{i)}: See Proposition 7.4 of \cite{R}. 

\emph{ii)}-\emph{iii)}: Use the fact that the Fourier transform maps $\S(\Rr)$  isomorphically into itself and the Cauchy-Schwarz inequality.
The homomorphism property follows from the linearity of the Fourier transform, the fact that pointwise multiplication is converted into 
convolution and the identity $e^{itD} = e^{i s D} e^{i (t - s) D}$, which follows from the uniqueness of solutions of the wave equation. The inequality
follows by a reduction to the case of compact manifolds.
We refer to the proof of Proposition 9.20 in \cite{R}. 

\emph{iv)}: For the finite propagation speed property we refer to Section 3.C in \cite{Ch} as well as the proof of Proposition 7.20 in  \cite{R}. See also \cite{R2}.
The essential self-adjointness follows from the quoted theorem of Chernoff, Theorem \ref{Thm:Chernoff}. 
\end{proof}

The following theorem is the generalization of the theorem for the tangent groupoid given by Siegel \cite{S}, Corollary 2
and Roe \cite{R}, Proposition 5.30, 5.31.


Recall first the following notions. An \emph{equivariant bundle} $\Ee \to \G$ over a Lie groupoid $\G$ is a vector bundle
such that $R_{\gamma} \colon \G_{r(\gamma)} \to \G_{s(\gamma)}$ induces a vector bundle isomorphism $R_{\gamma}^{\ast} \colon \Ee_{\G_{r(\gamma)}} \to \Ee_{\G_{s(\gamma)}}$. 
Given a vector bundle $E \to M$ we define $\Hom(E)$ to be the pullback bundle $r^{\ast}(E) \otimes s^{\ast}(E^{\ast})$ obtaining
a bundle $\Hom(E) \to \G$ over $\G$. We denote accordingly the bundle $\Hom(W) \to \Gad$ which is the equivariant lifting of the homomorphism bundle to the adiabatic groupoid. 
Denote by $\P$ the set of functions in the Schwartz class $\S(\Rr)$ which have compactly supported Fourier transform.

\begin{Thm}
Let $(M, \A, \V)$ be a Lie manifold with Clifford module $W \to M$ and $\G \rightrightarrows M$ a Lie groupoid such that $\A(\G) \cong \A$.
Denote by $\Dd := (\Dirac_{x, t})_{(x, t) \in M \times I}$ an equivariant family of geometric Dirac operators associated to $\nablaslash^W$ on $\Gad$.
Then there exists a ring homomorphism 
\[
\Psi_{\Dd} \colon C_0(\Rr) \to C_r^{\ast}(\Gad, \Hom(W))
\]
such that the regular representation
\[
\pi_{x, t} \colon C_c^{0,\infty}(\Gad, \Hom(W)) \to \L(L^2(\G_{(x,t)}^{ad}))
\]
fulfills the identity 
\[
 \pi_{x,t}(\Psi_{\Dd}(f)) = f(\Dirac_{x,t}),\quad f \in \P. 
\]
\label{Thm:fctcalc}
\end{Thm}

\begin{proof}
Applying Proposition \ref{Prop:fctcalc}(i) we fix the solution operator $e^{i \tau \Dirac_{x,t}}$ to the wave equation for
$\Dirac_{x,t}$. For given $f \in \P$ we use the functional calculus to define $f(\Dirac_{x,t}) = \frac{1}{2\pi} \int \hat{f}(\tau) e^{i \tau \Dirac_{x,t}}\,d\tau$.
By Proposition \ref{Prop:fctcalc}(iv), $f(\Dirac_{x,t})$ is a smoothing operator with finite propagation speed.
The equivariant family $(f(\Dirac_{x,t}))$ has a reduced kernel which we denote by $k^f$, obtained from the equivariant family of Schwartz kernels $k_{x,t}^f$ defined on the fibers and smooth with regard to $(x,t)$.
In view of  the finite propagation speed and the compactness of $M$ the reduced kernel $k^f$ is a compactly supported distribution.
We therefore define $\Psi_{\Dd}(f)$ via the assignment $\gamma \mapsto k_{s(\gamma)}^f$ in $C_c^{0, \infty}(\Gad)$.
This is a continuous ring homomorphism by Proposition \ref{Prop:fctcalc}(iii). 
The equation for the regular representation follows because $k^f$ is the reduced kernel of the operator $f(\Dd)$
on $\Gad$. Recall the definition of the regular representation $\pi_{x,t} \colon C_c^{0,\infty}(\Gad, \Hom(W)) \to \L(L^2(\G_{x,t}^{ad}))$ given by
\[
\pi_{x,t}(f)(\xi)(\gamma, t) = \int_{\G_{x,t}^{ad}} f(\eta, t) \xi(\gamma \eta^{-1}, t)\,d\mu_{(x,t)}(\eta), \ \xi \in L^2(\Gad).
\]
This yields by definition of $\Psi_{\Dd}$ the $L^2$-action
\[
f(\Dirac_{x,t})g(\gamma) = \pi_{x,t}(\Psi_{\Dd}(f)) g(\gamma) = (\Psi_{\Dd}(f) \ast g)(\gamma),
\]
and hence the last identity is proven.
Since $\P$ is dense in $C_0(\Rr)$ and $C_r^{\ast}(\Gad)$ is the completion of $C_c^{0,\infty}(\Gad)$ with respect to the regular representation we obtain by the spectral theorem
that the map $\P \to C_c^{0,\infty}(\Gad)$ is continuous with regard to the $C_0(\Rr)$ norm. This shows that $\Psi_{\Dd}(f)$ belongs to $C_r^{\ast}(\Gad)$. 
\end{proof}

\begin{Rem}
In general it is desirable to have a functional calculus with values in a smaller subalgebra of $C_r^{\ast}(\G)$ which has favorable properties for particular analytical contexts, see also \cite{R0}. 
Therefore we consider later the algebra $\S(\G, \Hom(W))$ which contains the groupoid heat kernel for $t > 0$. Recall that $\S(\G)$ is contained in $C_r^{\ast}(\G)$ whenever $\G$ has polynomial growth, cf. \ref{Prop:Schwartz}.
We note also that the heat kernel on a Lie groupoid is longitudinally smooth by definition (cf. \cite[Section 2]{BGV}). In general it has been proven for particular cases of Lie groupoids that the reduced groupoid heat kernel is also transversally smooth, i.e. $k_t \in C^{\infty}(\G \times (0, \infty))$, see \cite{So}. It is still an open question  whether every Lie groupoid integrating a given Lie structure has a smooth heat kernel. 
This additional property of transversal smoothness is required in particular cases to show that the
renormalized index $\indV$ equals the Fredholm index, cf. \cite{M}. 
\label{Rem:fctcalc}
\end{Rem}

\section{Renormalized Super Trace}

\label{renormalization}



We will next define a class of rapidly decaying distributions on the Lie manifold. This will be the class
which contains the heat kernel and on which we will  define the renormalized super trace.

Let $g$ be any compatible metric on $M$, i.e. a bilinear form on $\A$ and $d = d_g$ the metric distance
induced by $g$. The interior $M_0$ of $M$ with the metric $g$ restricted to it is a complete Riemannian manifold by \cite{ALN2}.
We define the spaces $\presuper{\V}{S^{k,l}}(M) := \{f \in C_0(M \times M) : \|f\|_{k,l} < \infty\}$,
where the semi-norm system $\|\cdot\|_{k,l}$  is given for $k,l \in \Nn$ by
\[
\|f\|_{k,l} := \sup_{1 \leq i \leq l}\  \sup_{\overline{v} = (v_1, \cdots, v_l) \in \V^l, \ \|v_i\| \leq 1} \ \sup_{x,y \in M} |(1 + d(x,y))^k \omega_{\overline{v}, i}(f)(x,y)|. 
\]
For  $\overline{v} := (v_1, \cdots, v_l) \in \V^l$  each $v_i$ can be regarded as a first order differential operator in $\Diff_{\V}^1(M)$. We let $\omega_{\overline{v}, i}(f) := v_1 \cdots v_i f v_{i+1} \cdots v_l$. 
In the same way as in the proof of Proposition \ref{Prop:proj} we show that if 
 $l$ is fixed and $k_1 \geq k_2$, we have
$\|f\|_{k_1, l} \leq \|f\|_{k_2, l}$ and if $l_1 \geq l_2$ with $k$ fixed we have $\|f\|_{k,l_1} \leq \|f\|_{k, l_2}$.
Hence the spaces $\left(\presuper{\V}{S^{k,l}}(M)\right)_{(k,l) \in \Nn^2}$ form a dense projective system of Banach spaces.


\begin{Def}
The Schwartz space of rapidly decaying functions on the Lie manifold $(M, \A, \V)$ is defined as
the space $\SV(M)$ given by the projective limit
\[
\SV(M) := \varprojlim_{k,l \in \Nn} \presuper{\V}{S^{k,l}}(M).
\]
\label{Def:SV}
\end{Def}

In the definition of a generalized trace class for the given Lie structure we face the problem that the density induced by the Lie structure is not integrable as we approach the boundary. See e.g. \cite{L1} for the $b$-case and
the example below. The remedy is a regularization procedure.
Similarly, one could define the canonical (KV) trace, the Wodzicki residue trace and the $\V$-determinant, but this is outside the scope of the present work. 
 
\begin{Exa}
Consider the case of the $b$ vector fields $\V = \V_b$ on a manifold $M_0$ with cylindrical end $(-\infty, 0]_s \times Y$, see also \cite{L1} for further details on this special case. 
The Kondratiev transform $x = e^s$ furnishes a compactification $\widehat{M_0} = M$ to a manifold with boundary, where  $s \to -\infty$ corresponds to $x \to 0$. 
Close to the boundary we have the density $ds = \frac{dx}{x}$ with $\partial_s = x \partial_x$. 
The singular structure is encoded in a Riemannian metric $g$ (a compatible metric on the $b$-tangent bundle $\A^b \to M$) which is of product type close to the boundary
\[
g = ds^2 + h = \left(\frac{dx}{x} \right)^2 + h. 
\]

Note that $\frac{dx}{x}$ is not integrable over $[0,1]_x$ so that the heat kernel $e^{-t \Delta_g}$ is not of trace class. 
We therefore use the regularization by observing that for $\Re z > 0$ the function $x^z$ is integrable with regard to $\frac{dx}{x}$ over $[0,1]_x$. 
Hence $x^z e^{-t \Delta_g}$ is trace class and by setting
\[
G(f)(z) = \int_M x^z f \,dg, \ f \in C^\infty(M), \ \Re(z) > 0
\]
we define the $b$-trace as the regularized value of $G(f)(z)$ in $z = 0$. 
\label{Exa:Loya}
\end{Exa}

We denote  by $\VOmega^1 := \Omega^1(\A)$  the one-density bundle and set $\rho := \prod_{F \in \F_1(M)} \rho_F$ where $\rho_F$ is a boundary defining function of a given closed codimension one face $F$.

\begin{Def}
A Lie manifold $(M, \V, \A)$ is called \emph{renormalizable} if 
%
%
there is a minimal $k \in \Rr$ such that $G(f)(z)=\int_M \rho^z f $ defines a function  $G(f)$ holomorphic on $\Re (z) > k-1$, which extends meromorphically to $\Cc$.
Then we  define 
\[ \davintV f : \ \text{regularized value (zero order Taylor coefficient) at} \ z = 0 \ \text{of} \ G(f). 
\]
\label{VTr}
\end{Def}

The following class of Lie structures covers in particular the examples given in \ref{Exa:cuspidal}. 

\begin{Def}
A Lie structure $\V$ of a Lie manifold $(M, \V, \A)$ is called \emph{exact} if, near each face with boundary defining function $x_1$,  the Lie structure is generated by $\left\{x_1^{k_1}\partial_{x_1},  \ldots,  x_1^{k_n}\partial_{x_n}\right\}$. 
\label{Def:exact}
\end{Def}

\begin{Prop}
An exact Lie manifold $(M, \A, \V)$ is renormalizable. 
\label{Prop:exact}
\end{Prop}

\begin{proof}
By definition $\V \subset \V_b$. 
On an arbitrary boundary face $F \in \F_1(M)$ fix  local coordinates $\{x_1, \cdots, x_n\}$ in a small tubular neighborhood $[0, \epsilon)\times F $. 
Then $\{x_1 \partial_{x_1}, \partial_{x_2}, \cdots, \partial_{x_n}\}$ is a local basis of $\V_b$. 
Consider the boundary defining function $\rho_F \colon M \to \overline{\Rr}_{+}$ and let $\nu \colon (-\epsilon, \epsilon)\times F \iso \U \subset M$ be the isomorphism from 
the tubular neighborhood theorem such that 
\[
(\rho_F \circ \nu)(x_1,x') = x_1, \ (x_1, x') \in [0, \epsilon)\times F.
\]
By assumption, 
\[
\V_{|\U} = \mathrm{span}_{C^{\infty}(\U)} \{V_1, \cdots, V_n\},
\]
where $V_1 = x_1^{k_1} \partial_{x_1}, \ V_2 = x_1^{k_2} \partial_{x_2}, \cdots, V_n = x_1^{k_n} \partial_{x_n}$.   
Note that $k := \sum_{j=1}^n k_j$ is the \emph{degeneracy index} $k = k_F$ of $F$ and therefore invariantly defined.
Following \cite[Section 4]{LM}, \cite{LM2} we obtain a meromorphic extension $G(f) \colon \Cc^{\F_1(M)} \to \Cc$ with at most simple poles in $z_F = k - 1 - j$ for $j \in \Nn_0$. 
\end{proof}

If $A\in \Psi_\V^m(M)$, $m<-n$, is a pseudodifferential operator in the Lie calculus with Schwartz kernel density $k_A$ we define the renormalized trace 
$$\TrV(A) = \davintV  {k_A}_{|\Delta_\V}$$
as the renormalized integral over the diagonal $\Delta_\V$ in $M\times M$.

Using the canonical symplectic form $\omega_\V$ on $\A^*$ we can alternatively write
\[
\TrV(A) = \davintVA  a \omega_\V^n. 
\]
Here $a \in S_{cl}^{m}(\A^{\ast})$, $A=\op a$, such  that $\Ff^{-1} a = \kappa_A$ near  $\Delta_{\V}$ with the fiberwise Fourier transform $\Ff$ defined in \cite[Chapter 1.5]{Simanca}.
The correspondence 
\[
\davintVM \longleftrightarrow \davintVA
\]
is obtained via the Fourier transform identity
\[
f(0) = \int_{\Rr^n} \F(f)(\zeta) \,d\zeta
\]
which is  applied fiberwise.

For exact Lie structures the finite part integral $\avintVA$ could  also be defined as in \cite{LM}, i.e. we have two interpretations
\begin{itemize}
\item Using growth conditions at infinity. 

\item Via radial compactification of $\A^{\ast}$ (i.e. compactify $\A^*$ to a manifold with corners $\widehat{\A} \to M$ which is fibered over $M$
such that $\widehat{\A}_x$ is a closed ball of dimension $n$). 

\end{itemize}



We will show later that the heat kernel of a generalized Laplacian on a Lie manifold is actually contained in the 
class $\SV(M)$. By the above example this does however not imply that the heat kernel is of trace class.
Nevertheless, we readily see that $\SV(M) \subset \LV(M)$ for $\LV(M)$ denoting the class of operators with bounded 
renormalized trace, i.e. the renormalized trace class.
Hence the renormalized trace extends to a well-defined functional on the class of rapidly decaying functions $\TrV \colon \SV(M) \to \Cc$.  


Assuming that the Lie manifold $(M, \V, \A)$ is given a spin structure $S \to M$ and Clifford module $W \to M$ we recall next the definition of the supertrace functional, which in our case acts fiberwise on the homomorphism bundle $\hom(W) \to M$.
Decompose the spinor bundle $W = W^{+} \oplus W^{-}$ into elements of even and odd degree. 
Suppose that the Dirac operator is odd graded with regard to this decomposition. 
According to \cite[1.3]{BGV}  the bundle is realized as a super-bundle (a bundle consisting of super spaces, i.e.
$\Zz_2$ graded spaces). 
The super bundle $\hom(W)$ decomposes as 
\begin{align*}
\hom(W) &= \begin{pmatrix} \hom(W^{+}, W^{+}) & \hom(W^{+}, W^{-}) \\
\hom(W^{-}, W^{+}) & \hom(W^{-}, W^{-}) \end{pmatrix}.
\end{align*}

Likewise, each element $T \in \hom(W)$ decomposes
\begin{align*}
T &= \begin{pmatrix} T^{++} & T^{+-} \\
T^{-+} & T^{--} \end{pmatrix}.
\end{align*}

We note that $\hom(W)_x = \hom(W_x, W_x) \cong \Cl(\A_x \otimes \Cc) \otimes \End_{\Cl}(W_x)$; hence $\hom(W)$ can be viewed as a bundle of
superalgebras. 
Given a super algebra $A$ denote by $[\cdot, \cdot]_s \colon A \times A \to A$ the supercommutator given by 
$[a,b]_s := ab - (-1)^{|a| |b|} ba$. 
Here $|a|\in \{0,1\}$ is the parity of $a$.
By definition, a supertrace is a linear form $\tr_s \colon A \to \Cc$ such that $\tr_s[a,b]_s = 0$. 
In our context we define $\tr_s \colon \hom(W) \to \Cc$ by $\tr_s(T) := \tr(T^{++}) - \tr(T^{--})$.
This yields a supertrace by Proposition 1.31 of \cite{BGV}.
Given a $\Cl(\A)$-module $W\to M$ we can define
the renormalized supertrace of an operator 
$T \in \SV(M, \Hom(W))$ with convolution kernel $k_T$, acting fiberwise on the graded bundle $\Hom(W)$ by
 $$\TrsV(T) = \davintV \tr_s(k_T)\,d\mu.$$

\medskip

Denote by $\Cl_0 \subseteq \Cl_1 \subseteq \cdots \subseteq \Cl_n(\A \otimes \Cc)$ the Clifford filtration by degree.
In \cite[Proposition 11.5]{R} Roe showed the following lemma  which we will need later for the construction of a suitable rescaling:

\begin{Lem}
We have $\tr_{s|\Cl_{n-1}} = 0$ and $\tr_s(e_1 \cdots e_n) = (-2i)^{\frac{n}{2}} $ for any oriented orthonormal frame $\{e_i\}_{i=1}^{n}$ of $\A$.
\label{Lem:Cl}
\end{Lem}







\section{Rescaling}



% ...


In this section we will  finish the proof of the index formula given in Theorem \ref{Thm:locind}.
We write the deformed Dirac operator on the adiabatic groupoid using the parametrization of Lie groupoids as defined in \cite{LR}. The Lichnerowicz theorem yields an expression for $\Dirac_{x,t}^2$ in normal coordinates.
Then we calculate the renormalized super-trace and extract the right coefficient using the rescaling as defined previously. First, however, we shall establish the following representation theorem:

\begin{Thm}
Let $(M, \A, \V)$ be a Lie manifold. If $\G \rightrightarrows M$ is a corresponding integrating Hausdorff Lie groupoid such that $\G_{|M_0}\cong M_0\times M_0$, then there exists a canonical isomorphism $\SV(M) \cong \S(\G)$ implemented by the vector representation $\varrho$.
\label{Thm:repr}
\end{Thm}

\begin{proof}
The surjectivity of $\varrho$ follows similarly as in the proof of Theorem 3.2. in \cite{ALN}. Here we neither need the Hausdorff condition nor the assumption that $\G$ restricts to the pair groupoid over the interior.  
We give some details for the benefit of the reader. The vector representation is given by the map $\varrho \colon \S(\G) \to \SV(M)$ with the 
defining property $(\varrho(T)\varphi) \circ r = T (\varphi \circ r), \ \varphi \in C^{\infty}(M)$.
Note first that the vector representation $\varrho \colon \S(\G) \to \SV(M)$ is a well-defined homomorphism. This follows from the definition since the reduced
metric distance is obtained from the $\G$-invariant family of metric distances $(\dd_x)_{x \in \Gop}$ on the fibers.
We obtain a family of $\S(\G_x)$ of Schwartz spaces. Additionally, the vector fields in $\V = \Gamma(\A)$ are in one-to-one
correspondence with the $\G$-invariant vector fields on the vertical tangent bundle $T^s \G$.

The range map $r$ is a submersion, hence it is locally of product type i.e. 
$T(\varphi \circ r) = \varphi_0 \circ r$ for some $\varphi_0 \in C_c^{\infty}(M_0)$. Then we have $\varrho(T) \varphi = \varphi_0$.
Let $x \in M_0$ be fixed. Also $r \colon \G_x \to M_0$ is a surjective local diffeomorphism.
The natural action of the isotropy group $\Gamma := \G_x^x$
on $\G_x$ is free. Therefore $r \colon\G_x \to M_0$ is a covering map with covering group $\Gamma$ and we have $\G_x / \Gamma \cong M_0$. 
Given $T \in \S(\G)$ we have by $\G$-invariance of $T$ that $T$ is in particular $\Gamma$-invariant.
We obtain that $T_x \colon C_c^{\infty}(\G_x) \to C^{\infty}(\G_x)$ descends to $\tilde{T} \colon C_c^{\infty}(M_0) \to C^{\infty}(M_0)$.
Define $r_{\ast}$ via $r_{\ast}(T_x) = \tilde{T}$. Since $T(\varphi \circ r) = \varphi_0 \circ r$ for a $\varphi_0 \in C_c^{\infty}(M_0)$, we obtain that $r_{\ast}(T_x) \varphi = \varphi_0$. 
Hence we can rewrite $\varrho$ in the form $\varrho(T) = r_{\ast}(e_x(T))$ where $e_x(T) = T_x$ is the evaluation at $x$. 
One can check that $e_x(T) = T_x$ is $\Gamma$-invariant, since $T$ is $\G$-invariant.
Thus we obtain the commuting diagram
\[
\xymatrix{
\S(\G) \ar@{->>}[d]_{e_x} \ar@{->>}[r]^{\varrho} & \SV(M) \\
\S(\G_x)^{\Gamma} \ar@{->>}[ur]_{r_{\ast}} &
}
\]

Since $e_x$ and $r_{\ast}$ are surjective it follows that $\varrho$ is surjective.

We prove the injectivity using the Hausdorff condition and the assumption that $\G$ restricts to the pair groupoid, see also \cite{N}.
Let $z \in M_0$ be fixed and denote by $e_z \colon \S(\G) \to \S(\G_z)$ the evaluation $T = (T_x)_{x \in M} \mapsto T_z$.
To see the injectivity of $e_z$ let $T_z = 0$.
We need to prove that $T_w = 0$ for each $w \in M$, i.e. $T = 0$. 
Since $\G_{|M_0} \cong M_0 \times M_0$ and the family $T$  is $\G$-invariant, it follows that $T_w = 0$ for each $w \in M_0$.
Let $w \in M$ be arbitrary, then $\scal{T}{\psi} = 0$ for each $\psi \in C^{\infty}(\G_w)$. 
In order to see this let $\varphi \in C^{0,\infty}(\G)$ be such that $\varphi_w = \psi$, which is possible since $\G_w \subset \G$
is closed in the locally compact Hausdorff space $\G$.
We choose a Haar system, then by the smoothness of the Haar system and the Hausdorff property of $\G$, the function
$w \mapsto \|\scal{T_w}{\varphi_w}\|$ is continuous and on $w \in M_0$ the function vanishes.
By density of $M_0$ in $M$ it follows $T_w \varphi_w = 0$ for each $w \in M$. Hence $e_z$ is injective.
The bijection $j \colon \S(\G_z) \to \SV(M)$ is obtained using the canonical diffeomorphism $\G_z \cong M_0$.
Since $\varrho$ equals $j \circ e_z$, it is injective. 
\end{proof}

From Theorem  \ref{Thm:rapdecay} we then obtain: 
\begin{Cor}
Let $(M, \A, \V)$ be a non-degenerate Lie manifold with spin structure $S \to M$ and Clifford module $W$ over $\Cl(\A)$. Denote by $D$ the Dirac operator induced by an admissible connection $\nabla^W$.
Then the heat kernel of the generalized Laplacian $e^{-t D^2}$ is contained in $\SV(M, \Hom(W))$ for $t > 0$ .
\label{Cor:repr}
\end{Cor}

\begin{Rem}
Given a Lie manifold $(M, \A, \V)$ assume that the Lie groupoid $\G \rightrightarrows M$ such that $\A(\G) \cong \A$ is Hausdorff and in addition has a length function of polynomial growth.
Then by Theorem \ref{Thm:repr} together with Proposition \ref{Prop:Schwartz} we obtain that the Lie calculus
$\Psi_{\V}^{m}(M) + \SV(M)$, with the smoothing ideal given by the Schwartz class $\SV(M)$, is closed under holomorphic functional calculus.
This is therefore in particular true for the examples given in Example \ref{Exa:cuspidal}. 
\label{Rem:repr}
\end{Rem}

We introduce the rescaled bundle and the method of extracting the right coefficient in the asymptotic expansion
Ansatz for the heat kernel. 
As usual $\Dirac$  denotes the Dirac operator on the groupoid $\G$ and  $D$  its vector representation, the Dirac operator on the Lie manifold $(M, \A, \V)$.

In the following we describe the structure of the rescaling approach to the local index theorem as explained in \cite{S} for the case of a closed smooth manifold. 
Assume we are given a non-degenerate Lie manifold with spin structure $S \to M$, a Clifford module $W \to M$ and let $\G \rightrightarrows M$ be an integrating
Lie groupoid which is Hausdorff. We obtain from the above the bundle $\Hom(W)=r^\ast(W)\otimes s^\ast(W) \to \Gad$ as a lifting.
Let $j \colon \A(\G) \hookrightarrow \Gad$ be the natural embedding as a submanifold. 
Denote by $\hom(W) \to M$ the bundle with fibers $\hom(W)_x = \hom(W_x, W_x) \cong \Cl(\A_x \otimes \Cc) \otimes \End_{\Cl}(W_x)$, $x \in M$.
Since on $\A(\G)$ source equals range we have
\[
\Hom(W)_{|\A} \cong j^{\ast} \hom(W) \cong \Cl(\A \otimes \Cc) \otimes \End_{\Cl}(W). 
\]

The basic idea for the definition of the rescaled bundle $\Ee \to \Gad$ is to extend a Clifford filtration by degree
to a neighborhood of $\A$ inside the adiabatic groupoid. More precisely, note the following.

\begin{itemize}
\item The \emph{rescaling} will be adapted to the Clifford filtration $\Cl_0 \subseteq \Cl_1 \subseteq \cdots \subseteq \Cl(\A \otimes \Cc)$ by degree, where the $\Cl_j$ are lifted to $\G$ using the range map.

\item The bundle $\Hom(W)$ is endowed with the connection  $(\mathrm{pr}_1 \circ s)^{\ast} \nabla$ via pullback: 
\[
\xymatrix{
\Gad \ar[r]^{s} & M \times I \ar[r]^{\mathrm{pr}_1} & M
}
\]
from the  Levi-Civita connection $\nabla$ on $M$.

\item We will extend the filtration $\{\Cl_k\}$ to a filtration $\{\tilde{\Cl}_k\}$ on a neighborhood of $\A$, see below.

\end{itemize}

We define the \emph{rescaled} sections
\[
\DV := \{u \in C_c^{\infty}(\Gad, \Hom(W)) : \nabla_{N}^{p} u_{|\A} \in C^{\infty}(\A, \Cl_{n-p} \otimes \End_{\Cl}(W)), \ 0 \leq p \leq n\}
\]
with the normal vector field $N = \partial_t$.

There is a bundle $\Ee \to \Gad$ such that $C_c^{\infty}(\Gad, \Ee) = i_{\Cl}^{\ast} \D$,
where $i_{\Cl} \colon \Ee \to \Hom(W)$ is a bundle map and an isomorphism over $\G_{|(0,1]}^{ad}$, see also Proposition 8.4 in \cite{M} for the case of $b$-vector fields. 
We fix the bundle $\Ee$ and refer to it as the \emph{rescaled bundle}. 

\begin{Prop}
An alternative description of $\DV$ is given by 
\begin{eqnarray}\label{eq:DV}
\DV = \left\{u \in C_c^{\infty}(\Gad, \Hom(W)) : u = \sum_{j=0}^n t^{n-j} u_j + t^{n+1} u' \ \text{near} \ \A\right\}
\end{eqnarray}
with $u_j \in C_c^{\infty}(\Gad, \tilde{\Cl}_{n-j}\otimes \End_{\Cl}(W))$ and $u' \in C_c^{\infty}(\Gad, \Hom(W))$, where $\tilde{\Cl}_{n-j}$ is the natural extension of the filtration $\Cl_{n-j}$ to a neighborhood of $\A$ in $\Gad$ using parallel transport via $N$.  
\label{Prop:taylor}
\end{Prop}

\begin{proof}
Write $F_j := \Cl_{n-j}\otimes \End_{\Cl}(W)$ for $1 \leq j \leq n$ and $E = \Hom(W)$. 
Using the generalized exponential map $\Exp \colon \A \to \G$ from Section \ref{section:IV} and the corresponding tubular neighborhood theorem we choose a neighborhood $U$ of  $\A$ in $ \Gad$.
We then extend the filtration $\{F_j\}$ to a filtration $\{\tilde{F}_j\}$ as follows: Any smooth section $v$ of $E$ over $\A$ can be extended in a unique way to  a section $\tilde{v} \in C^{\infty}(U, E)$
which is covariant constant along $\partial_t = N$ by solving the ODE
\begin{align}
\nabla_N \tilde{v} = 0, \ \tilde{v}_{|\A} = v. \label{ODE}
\end{align} 
Denoting by $v_l$ the coefficients of $v$ with respect to a local basis 
$\{e_l:l=1,\ldots,L\}$ of $E$, this can be rewritten locally in the form 
\[
\partial_t \tilde{v}_l + \sum_{m=1}^L \gamma_{lm} \tilde{v}_m=0, \ (\tilde{v}_l)_{|t=0} = v_l, 
\quad l=1,\ldots,L.
\]
Here, the $\gamma_{lm}$ are determined by $\nabla_N e_m = \sum_{l} \gamma_{lm} e_l.$	
 We set
\[
C^{\infty}(U, \tilde{F}_j) := \mathrm{span}_{C^{\infty}(U)} \{u : u \ \text{solution of \eqref{ODE} with initial data in }  F_j\}. 
\]
The set defined in \eqref{eq:DV} then is a subset of $\DV$ as defined before. 
Conversely writing the Taylor series of  $u \in C^{\infty}(U, E)$ at the boundary $\A$ to  order $n$ as $u = u_0 +  \cdots + t^n u_n + t^{n+1} u_{n+1}$, where  $u_j \in C^{\infty}(U, E), \ \nabla_N u_j  = 0, \ j! u_{j|\A} = (\nabla_N^j u)_{|\A}$, $j=1,\ldots,n$, and $u_{n+1} \in C^{\infty}(U, E)$, we obtain the assertion. 
\end{proof}

Finally, we want to extend the supertrace functional to the rescaled bundle $\Ee$.
Note first the following lemma.

\begin{Lem}
We have a canonical isomorphism of twisted Clifford algebras $$\Ee_{|\A} \cong \Lambda \A^{\ast} \otimes \End_{\Cl}(W).$$
\label{Lem:rescaled}
\end{Lem}

\begin{proof}
Note first that the filtration of the Clifford algebra $\Cl(\A \otimes \Cc)$ by degree has associated to it
a graded algebra which identifies with the exterior algebra $\Lambda \A^{\ast}$, cf. \cite{BGV}.
The rescaled bundle $\Ee$ associated to the filtration $\{\tilde \Cl_k\}$ by Clifford degree restricts to the graded bundle
associated to $\{\tilde \Cl_k\}$. By combining these two facts the assertion follows.
\end{proof}

A more direct proof would rely on the intuition that $\Ee$ is just the bundle obtained from $\Hom(W)$
by replacing over $t \not= 0$ the Clifford bundle $W$ (lifted to $\G$) with the Clifford bundle $W_t$
which is a $\G$-invariant bundle such that over each fiber $\G_x$ it is the Clifford bundle $W_x^t$ associated to the Riemannian
metric $t g_x(\cdot)$. 

From Lemma \ref{Lem:Cl} and the definition of the rescaling we obtain the following lemma.

\begin{Lem}
Let $\G_{\Delta} := \{\gamma \in \Gad : s(\gamma) = r(\gamma)\} \subset \Gad$. Then for $t \not= 0$ the supertrace functional $\tr_s $
maps $C_c^{\infty}(\G_{\Delta}, \Ee_{|\G_{\Delta}})$ to $t^n C_c^{\infty}(\G)$.
\label{Lem:Cl2}
\end{Lem}

\begin{proof}
By  Lemma \ref{Lem:Cl} the supertrace vanishes on $\Cl_j$, $j<n$, so \cite[Proposition 11.4]{R} in connection with our rescaling  yields the assertion. 
\end{proof}

The above lemma ensures that the right coefficient is extracted when we apply the supertrace functional to the vector representation
of the groupoid heat kernel. 
\medskip



Consider the Lie groupoid $\G \rightrightarrows \Gop$. We fix $x_0 \in \Gop$. Then a \emph{parametrization} of $\G$ at $x_0$ is given by 
a tuple $(\varphi, \psi)$, where $\varphi \colon U \to \Gop$ and $\psi \colon U \times V \to \G$ are homeomorphisms, $U$ 
is a $0$-neighborhood in $\Rr^n$ and $V$ is a $0$-neighborhood in $\Rr^m$.  
The following conditions should hold:
\begin{enumerate}\renewcommand{\labelenumi}{(\roman{enumi})}
\item 
$\psi(0,0) = x_0$, 
\item
$r(\psi(u,v)) = \varphi(u)$, 
\item
$\psi(U \times \{0\}) = \psi(U \times V) \cap \Gop$. 
\end{enumerate}
Note that $r$ is a submersion at $x_0$. Conditions \emph{ii)} and \emph{iii)} imply $\varphi(u) = \psi(u, 0)$.

This induces a parametrization of $\A(\G)$, more precisely of the neighborhood $\A(\G)_{\varphi(U)}$ of the fiber $\A_{x_0}(\G)$,
which is given by $\theta \colon U \times \Rr^m \to \A(\G)$, $\theta(u, v) = \left(\varphi(u), \frac{\partial \psi}{\partial v}(u,0) v\right)$.


Then  $\alpha = \psi \circ \theta^{-1}$
implements a diffeomorphism of the neighborhood of $(x_0, 0)$ given by $\theta(U \times V)$ with $\psi(U \times V)$.
Additionally, $\alpha(\A_x(\G)) \subset \G_x$ holds for each $x \in \varphi(U)$.

Conversely, by choosing $\alpha$ as the exponential map $\Exp \colon \A(\G) \to \G$ defined in  Section \ref{section:IV}, we find $\varphi$ and $\psi$ with the above properties, see \cite[p.145]{NWX}.
Then $\alpha(\A_x(\G) \cap V) = \G_x \cap W$, $\alpha_x'(0) = \id_{\A_x(\G)}$ where 
$\alpha_x = \alpha_{|\alpha_x(\G) \cap W}$. 
\medskip


We are now in a position to give a proof of the main theorem.

\begin{proof}[Proof of Theorem \ref{Thm:locind}]

We apply the previously constructed continuous functional calculus $\Psi_{\Dd}$ adapted to the $\G$-invariant family 
of Dirac operators $(\Dirac_{x,t})_{(x,t) \in M \times I}$,
given by $\Dirac_{x,t} = t \Dirac_x$ and set $\Dd := (t \Dirac_x)_{(x, t) \in M \times I}$. 
For $f\in C_0(\Rr)$,  $\Psi_{\Dd}(f) \in C_r^{\ast}(\Gad, \Ee)$ by the construction of the functional calculus. Here $\Ee \to \Gad$ 
is the rescaled bundle introduced before Proposition \ref{Prop:taylor}. 
Recall the action of the functional calculus
\[
f(\Dirac_{x,t}) g(\gamma) = \pi_{x,t}(\Psi_{\Dd}(f)) g(\gamma) = (\Psi_{\Dd}(f) \ast g)(\gamma),
\quad g\in  L^2(\Gad_{x,t}).
\]
In our case this yields for $t>0$
\[
f(t \Dirac) g(\gamma) = \int_{\G_{s(\gamma)}} \Psi_{\Dd}(f)(\gamma \eta^{-1}) g(\eta) t^{-n} \,d\mu_{s(\gamma)}(\eta).
\]
Here the scaling factor $t^{-n}$ stems from the natural choice of Haar system on the adiabatic groupoid, cf. \cite[(6.8)]{LR}. 
For the function $f(x) = e^{-x^2}$ we see that the reduced kernel $\Psi_{\Dd}(f)$, as an operator on $\G$, is $t^n k_{t^2}$. Then we recall that by Theorem \ref{Thm:rapdecay} we have in fact $\Psi_{\Dd}(f) \in \S(\Gad, \Ee)$. 
Let  $l_t := \Psi_{\Dd}(e^{-x^2})_{|\G_{\Delta}}$ be the restriction to the diagonal in $\Gad$. Then $l_t(\gamma) = t^n k_{t^2}(\gamma)$ for $t \not= 0$ and $\gamma \in \G_{\Delta_t}$. For $D = \varrho(\Dirac)$ we obtain 
\[
\TrsV(e^{-t D^2}) = \davintV \tr_s(\kappa_t(x, x)) \,d\mu(x)
\]
where $\mu = \mu_g$ is the density defined by the fixed compatible metric $g$ and $\kappa_t$ denotes the heat kernel of $e^{-tD^2}$.
By the representation Theorem \ref{Thm:repr} and Theorem \ref{Thm:rapdecay} we obtain that $\kappa_t \in \SV(M)$.
Denote by $\tilde{l}_t$ the vector representation of $l_t$.
The equation above makes sense for $t \not= 0$. 
We have  that $\tr_s(\kappa_{t^2|\Delta})$ identifies with $t^{-n} \tr_s(\tilde{l}_t)$.
Since $t^{-n} \tr_s(\tilde{l}_t)$ by Lemma \ref{Lem:Cl2} extends smoothly to $t = 0$, we have $t^{-n} \tr_s(\tilde{l}_t) = \tr_s(\tilde{l}_0) + o(t)$.
From $\Psi_{\Dd}(e^{-x^2})_{|t \not= 0} = t^n k_{t^2}$ and $\Psi_{\Dd}(e^{-x^2})_{|t=0} = k_{u|u=1}$ we obtain
\begin{align*}
\TrsV(e^{-t D^2}) &= \davintV t^{-\frac{n}{2}} \tr_s(\tilde{l}_{t^{\frac{1}{2}}}) \,d\mu = \davintV \tr_s(\tilde{l}_0)\,d\mu + o(t^{\frac{1}{2}}).
\end{align*}
Hence we have reduced the task to calculating $\tr_s(\tilde{l}_0)$. 
We calculate the kernel $l_0$ on the groupoid using the Lichnerowicz theorem applied to the fibers of the integrating groupoid.
Denote by $\varphi \colon U \to \Gop, \ \psi \colon U \times V \to \G$ the above parametrization of $\G$ around a fixed $x_0 \in M$.
Recall that, by definition, $\alpha_x = \alpha_{|\A_x(\G) \cap V}$  is induced by the exponential map $\exp_x$ on the fiber $\G_x$.
Write $\alpha_x(\gamma) = (a_1, \cdots, a_m) =: a$ for the corresponding geodesic coordinates.
Consider the induced parametrization of $\Gad$ given by $\Phi \colon U^{ad} \times V \to \Rr^n \times \Rr^m \times \Rr$, where $U^{ad} = U \times \Rr$.
Restrict this map to the chart $V \times \{x\} \times \{t\}$ and call the restriction $\Phi_{x,t}$. 
An elementary calculation yields $\Phi_{x,t}(\eta) = \frac{1}{t} (\alpha_x(\eta) - a)$. 
Then the Lichnerowicz theorem on the complete manifold $(\G_x, g_x)$ yields for $b = \Phi_{x,t}(\eta)$
\begin{eqnarray*}
\lefteqn{\Dirac_{x,t}^2 f(\eta) = \Dirac_{x,t}^2 f(\Phi_{x,t}^{-1}(b_1, \cdots, b_m))} \\
&=& t^2 \Dirac_x^2 f(\alpha_x^{-1}((tb_1, \cdots, tb_m) + a)) \\
&=& -t^2 \sum_{i} \left(\frac{1}{t} \partial_i^x  + \frac{1}{4} \sum_{j} \frac{1}{t} \left(R_{ij}^x(a_j + tb_j) \right) \right)^2 f(\eta)\\
&&+ \left(\sum_{i < j} F^{W_x / S}(e_i, e_j)(a_j + t b_j)(a_j + t b_j) + \frac{t^2}{4} \kappa\right) f(\eta)\\
&=& -\sum_{i} \left(\partial_i^x  + \frac{1}{4} \sum_{j} R_{ij}^x(a_j + t b_j) \right)^2 f(\eta)\\
&&+ \left(\sum_{i < j} F^{W_x / S}(e_i, e_j)(a_j + t b_j)(a_j + t b_j) + \frac{t^2}{4} \kappa\right) f(\eta).
\end{eqnarray*}

The right hand side depends smoothly on $t$ up to and including $t = 0$. In the limit as $t \to 0$ we obtain
\[
\Dirac_{x,0}^2 = -\sum_{i} \left(\partial_i^x + \frac{1}{4} \sum_{j} R_{ij}^x a_j\right)^2 + \sum_{i < j} F^{W_x / S}(e_i, e_j)(a_j)(a_j). 
\]
The remainder of the argument consists in the solution of the differential equation of the heat kernel of $\Dirac_{x,0}^2$, which one
recognizes as a harmonic oscillator with twisting. We can therefore use the analysis in \cite{BGV} to obtain the solution
in terms of a formal power series in the scalar curvature $R_{ij}^x$ and the exponential of the twisting bundle $\exp F^{W_x / S}$.
By the $\G$-invariance of the curvature tensor as well as the twisting curvature and the Lichnerowicz theorem
for Lie manifolds given in Theorem \ref{Thm:Lichnerowicz}, it follows from \cite{BGV}, p. 164 and \cite{R}, Proposition 12.25, 12.26 that we obtain the integrand $\Aroof \wedge \exp F^{S/W}$ 
in the trace formula. Thus we have shown that
\[
\lim_{t \to 0^+} \TrsV(e^{-tD^2}) = \davintV \Aroof \wedge \exp F^{W / S} \,d\mu.
\]
To obtain the limit $t \to \infty$ consider
\[
\lim_{t \to \infty} \TrsV(e^{-t D^2}) - \lim_{t \to 0^{+}} \TrsV(e^{-t D^2}) = \int_{0}^{\infty} \partial_t \TrsV(e^{-t D^2})\,dt.
\]

We have $\partial_t \TrsV(e^{-tD^2}) = \TrsV(\partial_t e^{-t D^2})$. 
The latter equals $- \frac{1}{2} \TrsV([D, D e^{-tD^2}]_s)$ since by the odd grading of $D$ we have $D^2 e^{-tD^2} = \frac{1}{2} [D, D e^{-tD^2}]_s$, where $[\cdot,\cdot]_s$ denotes
the super commutator. Setting $\etaV(D) := \frac{1}{2} \int_{0}^{\infty} \TrsV([D, D e^{-t D^2}]_s)\,dt$ this completes the proof of the index theorem.
\end{proof}

\section{The Fredholm index}

\subsection*{Fredholm conditions}

In this final section we will first address the question when the renormalized index equals the Fredholm index for Dirac operators which are Fredholm on the appropriate Sobolev spaces. 
Consider $(M, \A, \V)$ a non-degenerate spin Lie manifold with an integrating Lie groupoid $\G \rightrightarrows M$ which is Hausdorff and strongly amenable.
Let $W \to M$ be a Clifford module over $\Cl(\A)$ and $\Psi_{\V}^m(M; W)$ the calculus of pseudodifferential operators
on the Lie manifold, acting on sections of the bundle $W$, cf.~\cite{ALN}.
Let $r^{\ast} W \to \G$ be the pullback bundle along the range map of the groupoid. By the representation Theorem \cite[Theorem 3.2]{ALN} and \cite{N} the vector representation
furnishes a canonical isomorphism $\Psi_{\V}^m(M; W) \cong \Psi^m(\G; r^{\ast} W)$ with the groupoid pseudodifferential calculus. 
To any closed hyperface $F \in \F_1(M)$ of $M$ we associate a restriction homomorphism $\R_F \colon \Psi^m(\G; r^{\ast} W) \to \Psi^m(\G_F; r^{\ast} W_{|F})$
given by $P = (P_x)_{x \in M} \mapsto (P_x)_{x \in F}$. The combined homomorphism $\R(P) := \oplus_{F \in \F_1(M)} \R_F(P)$ is called the \emph{indicial symbol}. 

\begin{Def}
A pseudodifferential operator $P \in \Psi_{\V}^m(M; W)$ is \emph{fully elliptic} if its principal symbol $\sigma_m(P)$ and its indicial symbol $\R(P)$ are both pointwise invertible.
\label{Def:fullyelliptic}
\end{Def}

A Lie groupoid will be called \emph{strongly amenable} if the natural representation of its $C^{\ast}$-algebra $C^{\ast}(\G)$
on $\L(\H)$ is injective for $\H = L^2(M_0)$. We have the following auxiliary result from \cite{N}.
\begin{Thm}
Let $(M, \A, \V)$ be a Lie manifold such that there is a Lie groupoid $\G \rightrightarrows M$ that is Hausdorff and strongly amenable such that $\A(\G) \cong \A$.
A pseudodifferential operator $P \in \Psi_{\V}^m(M; W)$ is Fredholm $P \colon H_{\V}^{s}(M; W) \to H_{\V}^{s-m}(M; W)$ if and only if
it is fully elliptic.
\label{Thm:fullyelliptic}
\end{Thm}

We are now in a position to give a proof of Theorem \ref{Thm:Fh}. 
\begin{proof}[Proof of Theorem \ref{Thm:Fh}]
We fix the degeneracy index $k$ of $(M, \A, \V)$, cf. Definition \ref{VTr}. 
Denote by $\S_{tr} = \rho^k\,  \SV(M; W)$ the ideal of all elements of $\SV(M; W)$ vanishing to order $k$ at the boundary strata of $M$, and consider the induced 
short exact sequence
\[
\xymatrix{
\S_{tr} \ar@{>->}[r] & \SV(M; W) \ar@{->>}[r]^-{\Rk} & \S
}
\]
with the quotient $\S$ and the quotient map $\Rk$. 
The result follows by an application of the argument from \cite[Section 1.2]{LMP2} to the above short exact sequence.
We include some details for the convenience of the reader. The supertrace is a linear functional $\Tr_s \colon \S_{tr} \to \Cc$ such that $\Tr_s(ax) = \Tr_s(xa)$ for all $a \in \SV(M; W)$ and $x \in \S_{tr}$. 
The 
renormalized super-trace $\TrV_s \colon \SV(M; W) \to \Cc$ is by definition a linear extension of $\Tr_s$. 
Since $D$ is assumed to be fully elliptic, it is a Fredholm operator by Theorem \ref{Thm:fullyelliptic}. 
The projection onto the kernel of $D$ is contained in $\S_{tr}$, i.e. $P_{\ker D} \in \S_{tr}$ since $\SV$ is spectrally invariant by Proposition \ref{Prop:Schwartz}. 
Define the cyclic cocycle $\mu$ via $\mu(\Rk(a_0), \Rk(a_1)) = \TrV_s[a_0, a_1]$ for $a_0, a_1 \in \SV(M; W)$. 
In particular we define 
\begin{align*}
A_0(t) := D e^{-\frac{1}{2} t D^2}, \\
A_1(t) := \int_{t}^{\infty} D e^{-(\frac{t}{2} - s) D^2} \,ds.
\end{align*}

It can be checked that $A_0(t), \ A_1(t) \in \SV(M; W)$ for $t > 0$, cf. Theorem \ref{Thm:rapdecay}.
Denoting by $b \colon C_{\lambda}^1(\S) \to C_{\lambda}^2(\S)$ the boundary map in cyclic cohomology we check that 
\begin{align}
& b(A_0 \otimes A_1) = 2 \int_{t}^{\infty} D^2 e^{-s D^2}\,ds = 2 (e^{-t D^2} - P_{\ker D}). \label{bdycycle} 
\end{align}

Denote by $\brel, \ \Brel$ the boundary map in relative cyclic cohomology, i.e. 
\[
\brel = \begin{pmatrix} b & -\Rk^{\ast} \\ 0 & - b \end{pmatrix}, \ \Brel = \begin{pmatrix} B & 0 \\ 0 & - B \end{pmatrix}.
\]

We notice that $(\TrV_s, \mu)$ yields a relative cyclic cocycle with $\brel\begin{pmatrix} \TrV_s \\ \mu \end{pmatrix} = 0$. 
With \eqref{bdycycle} we obtain
\[
\Rk(e^{-t D^2}) = \Rk(e^{-t D^2} - P_{\ker D}) = \frac{1}{2} b(\Rk(A_0) \otimes \Rk(A_1))
\]

which yields
\[
\brel \begin{pmatrix} e^{-t D^2} \\ -\frac{1}{2} \Rk(A_0) \otimes \Rk(A_1) \end{pmatrix} = 0.
\]

We define the class in relative cyclic homology $\mathrm{EXP}_t(D) := (e^{-tD^2}, -\frac{1}{2} \Rk(A_0) \otimes \Rk(A_1))$. 
Denote by $\relb = \begin{pmatrix} b & 0 \\ -\Rk_{\ast} & -b \end{pmatrix}$ the boundary map in cyclic homology.
Then we have
\[
\relb \begin{pmatrix} \frac{1}{2} A_0 \otimes A_1 \\ 0 \end{pmatrix} = \begin{pmatrix} e^{-tD^2} - P_{\ker D} \\ -\frac{1}{2} \Rk(A_0) \otimes \Rk(A_1) \end{pmatrix}.
\]

We obtain a class $[\mathrm{EXP}_t(D)] = [(P_{\ker D}, 0)]$ in $\H_0^{\lambda}(\SV(M; W), \S)$. The latter equals $[P_{\ker D}]$ as a class in periodic cyclic homology $\HP_0(\S_{tr})$ using excision.
Altogether we have $[\mathrm{EXP}_t(D)] = [(P_{\ker D}, 0)]$ in $\HC_0^{\lambda}(\SV(M; W), \S)$, hence is independent of $t$.
The pairing of cyclic homology with cyclic cohomology and the definition of the Fredholm index yields
\begin{align*}
\ind D &= \Tr_s(P_{\ker D}) = \scal{(\TrV_s, \mu)}{(e^{-t D^2}, -\frac{1}{2} \Rk(A_0) \otimes \Rk(A_1))} \\
&= \TrV_s(e^{-t D^2}) - \frac{1}{2} \mu(\Rk(A_0) \otimes \Rk(A_1)) \\
&= \TrV_s(e^{-t D^2}) - \frac{1}{2} \TrV_s([A_0, A_1]). 
\end{align*}

Since $\TrV_s[A_0, A_1] = \TrV_s\left(\int_t^{\infty} D^2 e^{-s D^2} \,ds\right) = 2 \etaV(D)$ the assertion follows.
\end{proof}

\subsection*{Applications}

We turn now to applications of the index theorem \ref{Thm:locind} to the geometry of non-compact manifolds.
In the following we fix a Lie manifold as in the beginning of this section satisfying the assumptions of Theorem \ref{Thm:Fh}. 
The first observation follows by a well-known 
application of the Lichnerowicz formula \ref{Thm:Lichnerowicz}.

\begin{Thm}[Lichnerowicz vanishing]
Assume there is a compatible metric $g = g_{\A}$ of positive scalar curvature $\kappa > 0$, then $\ind(D) = 0$ for any spin Dirac operator $D = D^S$.
\label{Thm:vanishing}
\end{Thm}

\begin{proof}
Assume for the sake of a contradiction $D^2 s = 0$ for some $s \not= 0$. The Lichnerwicz formula for spin Dirac operators without twisting yields
\[
\frac{1}{4} \kappa = \scal{D^2 s}{s} - \|\nabla^{\ast} \nabla s\|^2 = -\|\nabla^{\ast} \nabla s\|^2 \leq 0. 
\]

This is a contradiction and hence $\ker D^2 = \{0\}$ and, in particular, $\ind(D) = 0$.
\end{proof}

A direct application of Theorem \ref{Thm:locind} combined with \ref{Thm:vanishing} yields the following result.
\begin{Cor}
Assume that there is compatible metric $g$ with positive scalar curvature, then 
\[
\etaV(D) = -\davintVM \Aroof \,d\mu
\]

for any spin Dirac operator $D = D^S$.
\label{Cor:obstruction}
\end{Cor}

We discuss another obstruction to the existence of positive scalar curvature metrics. 
As a preperation we recall that in the pseudodifferential calculus $\Psi_{\V}^m(M; W)$ in order $m=0$ there is the principal symbol $\sigma \colon \Psi_{\V}^0(M; W) \to C^{\infty}(S^{\ast} \A)$
and the \emph{indicial symbol} $\R \colon \Psi_{\V}^0(M; W) \to \oplus_{F \in \F_1(M)} \Psi_{\V(F)}^0(F; W_{|F})$. 
We have a short exact sequence by \cite{N}
\begin{align}
\xymatrix{
\K_{M_0} \ar@{>->}[r] & \Psi_{\V}^0(M; W) \ar@{->>}[r]^-{\sigma \oplus \R} & \SigmaV.
} \label{FhSES}
\end{align}

Here $\K_{M_0} = \K(L^2(M_0))$ denotes the compact operators on $M_0$. By $\SigmaV$ we denote the restricted direct sum
of $C^{\infty}(S^{\ast} \A)$ and $\oplus_F \Psi_{\V(F)}^0(F; W_{|F})$, i.e. a pullback. The situation is most easily summarized in terms of the pullback diagram together with
the full symbol $\sigma \oplus \R$ and restriction maps $(r_F)_{F \in \F_1(M)}$:
\[
\xymatrix{
\SigmaV \ar@{->}[d(1.8)]_{\pi_2} \ar@{->}[r(1.5)]^-{\pi_1} & &  C^{\infty}(S^{\ast} \A) \ar[d(1.8)]^{(r_F)_{F \in \F_1(M)}} & & \\
& \ar@{->>}[ul]_{\sigma \oplus \R} \ar@{->>}[dl]_{\R} \Psi_{\V}^0(M;W) \ar@{->>}[ur]^{\sigma} \ar@{->>}[dr]_{} & \\
\oplus_{F} \Psi_{\V(F)}^{0}(F;W_{|F}) \ar@{->>}[r(1.4)]^-{\oplus_{F} \sigma_F} & &  \oplus_{F} C^{\infty}(S^{\ast} \A_{|F}) & &
}
\]

For any Fredholm pseudodifferential operator $P \in \Psi_{\V}^0(M; W)$ we see from  \eqref{FhSES} that the index $\ind(P)$ only depends on the full symbol.
Since in general the $K$-theory group of the symbol algebra $\SigmaV$ is difficult to understand we consider next a deformation groupoid called \emph{Fredholm groupoid}, see e.g. \cite{CLM}. 
Start with the adiabatic groupoid $\Gad = \A(\G) \times \{0\} \cup \G \times (0, 1] \rightrightarrows M \times [0,1]$ with the topology as defined in Section \ref{section:IV}.
Setting $M_{\F} := (M \times [0,1]) \setminus (\partial M \times \{1\})$ we consider the subgroupoid $\G_{\F} \rightrightarrows M_{\F}$
which is written as a set $\G_{\F} = \A(\G) \times \{0\} \cup \G_{|\partial M} \times (0,1) \cup M_0 \times M_0 \times (0, 1]$. 
The groupoid $\G_{\F}$ is by construction a sub Lie groupoid of $\Gad$.
We refer to $\G_{\F}$ as the \emph{Fredholm groupoid}. Note that $M_0 \times M_0 \times (0,1]$ is a saturated open dense
subgroupoid of $\G_{\F}$ and set $\Tau := \G_{\F} \setminus (M_0 \times M_0 \times (0,1])$. 
Setting $M_{\partial} := M_{\F} \setminus (M_0 \times (0,1])$ we obtain again a Lie groupoid $\Tau \rightrightarrows M_{\partial}$.
As a set $\Tau$ is given by $\A(\G) \times \{0\} \cup \G_{\partial M} \times (0,1)$. It is the latter groupoid which is the natural home for the
full symbol of a pseudodifferential operator in the groupoid or the Lie pseudodifferential calculus.
Denote by $e_{0} \colon C^{\ast}(\G_{\F}) \to C^{\ast}(\Tau), \ e_1 \colon C^{\ast}(\G_{\F}) \to C^{\ast}(M_0 \times M_0 \times \{1\})$ the evaluations. 
Note that since $\G$ is (strongly) amenable by assumption we obtain that $\G_{\F}$ is (strongly) amenable as well.
We have the short exact sequence
\[
\xymatrix{
C^{\ast}(M_0 \times M_0) \otimes C(0,1] \ar@{>->}[r] & C^{\ast}(\G_{\F}) \ar@{->>}[r]^{e_0} & C^{\ast}(\Tau). 
}
\]

The $C^{\ast}$-algebra $C^{\ast}(M_0 \times M_0)$ is $\ast$-isomorphic to the compact operators on $M_0$, i.e. $\K_{M_0} = \K(L^2(M_0)) \cong C^{\ast}(M_0 \times M_0)$. 
Note that the $C^{\ast}$-algebra $C^{\ast}(M_0 \times M_0) \otimes C(0,1]$ is contractible and hence the induced map in $K$-theory
$(e_0)_{\ast} \colon K_0(C^{\ast}(\G_{\F})) \to K_0(C^{\ast}(\Tau))$ is invertible.
Setting $\ind_{\F} := (e_1)_{\ast} \circ (e_0)_{\ast}^{-1} \colon K_0(C^{\ast}(\Tau)) \to \Zz$ we have that for any fully elliptic pseudodifferential
operator $P \in \Psi_{\V}^{0}(M; W)$ the Fredholm index equals $\ind(P) = \ind_{\F}(P)$. We refer to \cite{CLM}. 
By the exact sequence \eqref{FhSES} the Fredholm index only depends on the full symbol of the given operator. We can therefore
consider the $K$-theory class in $K_0(C^{\ast}(\Tau))$. Given a first order spin Dirac operator $D = D^S$ with full symbol $a = (\sigma \oplus \R)(D)$
we can associate a class $[a] \in K_0(C^{\ast}(\Tau))$ as follows. Apply the order reduction which is contained in the completion of the Lie calculus $\overline{\Psi}_{\V}^1(M; S)$ as defined in \cite{N}
to the operator $D$ to obtain a zero order operator $\tilde{D}$. The class $[a] \in K_0(C^{\ast}(\Tau))$ is defined as the $K$-theory class of
the full symbol $(\sigma \oplus \R)(\tilde{D})$. 

We introduce next a secondary invariant which encodes information in $K$-theory about the structure of a given compatible metric of positive scalar curvature.
See also \cite{Z} where secondary invariants are introduced that control the vanishing of the generalized index defined via the adiabatic groupoid, instead of the Fredholm index defined via the Fredholm groupoid.
Consider first the following general setup: Let $A$, $B$ be separable $C^{\ast}$-algebras. We work inside the category $KK$ which consists 
of separable $C^{\ast}$-algebras as objects and elements of $KK(A, B)$ as arrows between objects. We denote these arrows by $A \dasharrow B$. The composition is given by the Kasparov product.
For a given \emph{connecting morphism} $\partial \in KK(A, B)$, then there is a $C^{\ast}$-algebra $A'$, a $\ast$-homomorphism $\varphi \colon A' \to A$ which is a $KK$-equivalence and a $\ast$-homomorphism $\psi \colon A' \to B$.
The situation is summarized in terms of the diagram:
\[
\xymatrix{
A' \ar@{-->}[d]_-{[\varphi]} \ar@{-->}[dr]^-{[\psi]} & \\
A \ar@{-->}[r]_-{\partial} & B.
}
\]

We can therefore write $\partial = [\varphi]^{-1} \otimes_{A'} [\psi]$, i.e. in terms of the Kasparov product of the corresponding classes in $KK$.
Fix the mapping cone of $\psi$  
\[
C_{\psi}(A', B) = \{a \oplus f : f(0) = \psi(a)\} \subset A' \oplus C([0,1), B). 
\]

Set $E := C_{\psi}(A', B)$ and $S := C_0(0,1)$ for brevity. Then $\partial$ is up to $KK$-equivalence the boundary map associated to the exact sequence.
\[
\xymatrix{
B \otimes S \ar@{>->}[r] & E \ar@{->>}[r]^-{q} & A'.
}
\]

If $[a] \in K_0(A')$, then $\partial[a] \in K_1(S \otimes B)$ is the \emph{primary invariant} whose vanishing is controlled by 
the \emph{secondary invariant} $\rho(a) \in K_0(E)$. Apply the short exact sequence in $K$-theory to obtain
\[
\xymatrix{
K_0(S \otimes B) \ar[r] & K_0(E) \ar@{->>}[r] & K_0(A') \ar[d]_{\partial = 0} \\
K_1(A') \ar[u] & \ar[l] K_1(E) & \ar[l] K_1(S \otimes B).
}
\] 

Hence we define $\rho(a)$ as the lift of $[a]$ to $K_0(E)$.

To apply the above construction of the class $\rho(a)$ to our situation we once again consider the Fredholm groupoid as previously introduced.
Also we specialize to the case $A' = C^{\ast}(\G_{\F}) ,\ B = \K, \ A = C^{\ast}(\Tau), \ E = C_{e_1}(C_r^{\ast}(\G_{\F}), C_r^{\ast}(M_0 \times M_0))$. 
Note that by the previous discussions $e_0 = \varphi$ induces a $KK$-equivalence and our primary invariant can be written in terms of the Kasparov product
\[
\ind_{\F} = [a] \otimes [e_0]^{-1} \otimes [e_1] \in K_0(\K) \cong \Zz. 
\]

Here $[a] \in K_0(C^{\ast}(\Tau)) \cong KK(\Cc, C^{\ast}(\Tau))$ is the class of the full symbol. Setting $\mathring{\G}_{\F} := \G_{\F|[0,1)}$ the short exact sequence in our case becomes
\[
\xymatrix{
C_r^{\ast}(M_0 \times M_0) \otimes C_0(0,1) \ar@{>->}[r] & C_r^{\ast}(\mathring{\G}_{\F}) \ar@{->>}[r]^{e_0} & C_r^{\ast}(\Tau). 
}
\]

By the invariance of the subgroupoid $M_0 \times M_0 \times \{1\}$ we have $K_0(C_{e_1}(C_r^{\ast}(\G_{\F}), C_r^{\ast}(M_0 \times M_0)) \cong K_0(C_r^{\ast}(\mathring{\G}_{\F}))$. 
We are now in a position to define the secondary invariant associated to compatible positive scalar curvature metrics on Lie manifolds (see also \cite{Schick}).

\begin{Def}
Given a compatible metric of positive scalar curvature $g$ the associated \emph{structure class} $\rho(M, g)$ is defined as the lifting of the full symbol of the spin Dirac operator $D = D_g$, i.e. $\rho(M, g) \in K_0(C_r^{\ast}(\mathring{\G}_{\F}))$. 
\label{Def:struct}
\end{Def}

With this $K$-theory class at hand we can give an alternative proof of Corollary \ref{Cor:obstruction} using $K$-theory.
\begin{proof}
Fix the positive scalar curvature metric $g$. Then in the six term exact sequence in $K$-theory 
\[
\xymatrix{
K_0(C_r^{\ast}(\mathring{\G}_{\F})) \ar[r] & K_0(C_r^{\ast}(\Tau)) \ar[r]^-{\partial} & K_1(\K \otimes C_0(0,1)) 
}
\]
we have $\rho(M, g) \mapsto [a] \mapsto \ind_{\F}(a) = 0.$ 
\end{proof}

The secondary invariant $\rho(M, g)$ contains information about the compatible metric of positive scalar curvature. Since the groupoids under consideration here are in particular Hausdorff and
amenable, a Baum-Connes type isomorphism furnishes the computability of the invariant. We refer to \cite{Tu} for the general proof of the Baum-Connes conjecture for amenable Hausdorff groupoids.
For concrete realizations of the Baum-Connes map and calculations of the corresponding $K$-theory of the $C^{\ast}$-algebras, in the special case of the $b$-type Lie structure, we refer to the recent work of Lescure and Carillo Rouse \cite{CL}. 


%

\section*{Acknowledgements}

The first author was supported through the programme \emph{Oberwolfach Leibniz Fellows} by  Mathematisches Forschungsinstitut Oberwolfach in 2015. 
We thank Paulo Carrillo Rouse, Jean-Marie Lescure, Victor Nistor and Bing Kwan So for useful discussions.

\end{document}